\documentclass[production,11pt,sort]{ip-journal}

\pubyr{2024}
\volno{0}

\IfFileExists{jdg.cfg}{%
    \input{jdg.cfg}%
    \loadnatbib
    \CRC
    \pubyear{2024}%
    \volume{0}%
    \issue{0}%
    \firstpage{1}%
    \lastpage{30}%
    }{}

\usepackage{mathtools}
\usepackage{hyperref}
\hypersetup{%
   colorlinks,
linkcolor=blue,
citecolor=blue,
urlcolor=blue
}
\makeatletter
\AtBeginDocument{%
 \@ifpackageloaded{hyperref}{\long \def\@secondoffour#1#2#3#4{#2}%
 \def\oldpageref#1{%
 \expandafter\old@setref\csname r@#1\endcsname\@secondoffour{#1}%
 }}{}%
 }
\mathtoolsset{}

\newtheorem{theorem}{Theorem}
\newtheorem{lemma}[theorem]{Lemma}
\newtheorem{proposition}[theorem]{Proposition}

\newcommand{\opvtex}{\operatorname{\Pi}}
\newcommand{\opp}{\operatorname{\Pi}^*\!}
\newcommand{\ofpp}[1]{\operatorname{\Pi^{\ast,{\it #1}}}\!}

\newcommand{\omr}[1]{\operatorname{R}_{#1}}

\newcommand{\Rvtex}{\mathbb R}
\newcommand{\Bvtex}{{B^n}}
\newcommand{\Svtex}{\Delta}

\newcommand{\snvtex}{{\mathbb S}^{n-1}}

\newcommand{\borel}{\mathcal{B} (\Rvtex^n)}
\newcommand{\proj}{\operatorname{proj}}

\newcommand{\dVlog}{\tilde V_{\log}}
\newcommand{\sgn}{\operatorname{sgn}}

\newcommand{\dvtex}{\,\mathrm{d}}

\newcommand{\Dvtex}{\mathrm{D}}
\newcommand{\vol}[1]{ \vert #1\vert}
\renewcommand \div {\operatorname{div}}

\renewcommand{\chi}{\operatorname{1}}

\allowdisplaybreaks
\begin{document}

\title[Affine fractional Sobolev inequalities]{Affine fractional Sobolev and isoperimetric inequalities} 

\author[J. Haddad ]{Juli\'an Haddad}
\address{Departamento de An\'{a}lisis Matem\'{a}tico\\ Facultad de Matem\'{a}ticas\\ Universidad de Sevilla\\ Sevilla, Spain\\}
\email{jhaddad@us.es}
\author[M. Ludwig]{Monika Ludwig}
\address{Institut f\"ur Diskrete Mathematik und Geometrie\\
Technische Universit\"at Wien\\
Wiedner Hauptstra\ss e 8-10/1046\\
1040 Wien, Austria\\}
\email{monika.ludwig@tuwien.ac.at}
\makeatother
%

\begin{abstract}
Sharp affine fractional Sobolev inequalities for functions on
$\mathbb R^{n}$ are established. For each $0<s<1$, the new inequalities
are significantly stronger than (and directly imply) the sharp fractional
Sobolev inequalities of Almgren and Lieb. In the limit as
$s\to 1^{-}$, the new inequalities imply the sharp affine Sobolev inequality
of Gaoyong Zhang. As a consequence, fractional Petty projection inequalities
are obtained which are stronger than the fractional Euclidean isoperimetric
inequalities, and a natural conjecture for radial mean bodies is proved.

\end{abstract}


\date{July 26, 2022} 

\maketitle


\section{Introduction}

Almgren and Lieb \cite{AlmgrenLieb} established P\'olya--Szeg\H o inequalities
for fractional Sobolev norms. The equality case was settled by Frank and
Seiringer \cite{FrankSeiringer}. A consequence is the fractional Sobolev
inequality (cf.~\cite{Ludwig:fracperi,FrankSeiringer}): for $0<s<1$, there
is $\alpha _{n,s}>0$ (with explicit value given in~\eqref{eq_alpha}) such
that for $f\in W^{s,1}(\mathbb R^{n})$,
%
\begin{equation}
\label{eq_fracsobo}
\Vert f\Vert _{{\frac {n}{n-s}}} \le \alpha _{n,s}\int _{\mathbb R^{n}}
\int _{\mathbb R^{n}}
\frac{\vert f(x)-f(y)\vert}{\vert x-y \vert ^{n+s}}\dvtex x\dvtex y,
\end{equation}
and there is equality if and only if $f$ is, up to sets of measure zero,
a constant multiple of the indicator function of a ball. Here
$\Vert f\Vert _{p}$ denotes the $L^{p}$ norm of $f$ and
$\vert \cdot \vert $ the Euclidean norm in $\mathbb R^{n}$ while
$W^{s,1}(\mathbb R^{n})$ denotes the fractional Sobolev space of
$L^{1}$ functions with finite right-hand side in~\eqref{eq_fracsobo}. By
a result of Bourgain, Brezis, and Mironescu~\cite{BourgainBrezisMironescu},
for $f\in W^{1,1}(\mathbb R^{n})$ with compact support,
%
\begin{equation}
\label{eq_BBM}
\lim _{s\to 1^{-}} \alpha _{n,s}\int _{\mathbb R^{n}} \int _{
\mathbb R^{n}} \frac{\vert f(x)-f(y)\vert}{\vert x-y \vert ^{n+s}}\dvtex x
\dvtex y = \frac {1}{n\omega _{n}^{1/n}}\int _{\mathbb R^{n}} \vert
\nabla f(x)\vert \dvtex x,
\end{equation}
where $\omega _{n}$ is the $n$-dimensional volume of the unit ball in
$\mathbb R^{n}$, and $W^{1,1}(\mathbb R^{n})$ denotes the Sobolev space
of $L^{1}$ functions $f$ with weak $L^{1}$ gradient $\nabla f$. D\'avila~\cite{Davila}
proved a result corresponding to~\eqref{eq_BBM} for compactly supported
$f\in BV(\mathbb R^{n})$, the space of $L^{1}$ functions with bounded variation,
and the assumption on the support can be dropped (see
\cite[Theorem 1.4]{AlbericoCianchiPickSlavikov}). Hence, the sharp
$L^{1}$ Sobolev inequalities on $W^{1,1}(\mathbb R^{n})$ and
$BV(\mathbb R^{n})$ follow from the fractional Sobolev inequalities~\eqref{eq_fracsobo}.

Gaoyong Zhang \cite{Zhang99} established a sharp affine Sobolev inequality
that is significantly stronger than the classical $L^{1}$ Sobolev inequality.
For $n\ge 2$ and $C^{1}$ functions $f: \mathbb R^{n}\to \mathbb R$ with
compact support, Zhang \cite{Zhang99} proved that
%
\begin{equation}
\label{eq_AffineSobolev}
\begin{split}
\Vert f\Vert _{{\frac {n}{n-1}}}
&\le \frac{ \omega _{n}}{ 2 \omega _{n-1}} \Big(\frac{1}{n} \int _{\snvtex} \Vert \langle \nabla f(\cdot ),\xi \rangle \Vert _{1}^{-n} \dvtex \xi \Big)^{-1/n}
\\
&\le \frac {1}{n\omega _{n}^{1/n}}\int _{\mathbb R^{n}} \vert \nabla f(x) \vert \dvtex x.
\end{split}
\end{equation}
Here $\langle \cdot ,\cdot \rangle $ denotes the inner product, and integration
on the $(n-1)$-dimensional unit sphere $\snvtex $ is with respect to the
$(n-1)$-dimensional Hausdorff measure. The first inequality in~\eqref{eq_AffineSobolev} is now called the affine Sobolev inequality or
the Sobolev--Zhang inequality. It is affine since both terms remain invariant
under volume-preserving affine transformations. Tuo Wang
\cite{Tuo_Wang} extended the Sobolev--Zhang inequality to $L^{1}$~functions
of bounded variation and showed that there is equality, up to sets of measure
zero, in this generalized Sobolev--Zhang inequality precisely for constant
multiples of indicator functions of ellipsoids.

The main aim of this paper is to establish affine fractional Sobolev inequalities
for $0<s<1$ that are stronger than the fractional Sobolev inequalities
in~\eqref{eq_fracsobo} and imply, in the limit as $s\to 1^{-}$, the Sobolev--Zhang
inequality.

\begin{theorem}%
\label{thm_afracsobo}
For $0<s<1$ and $f\in W^{s,1}(\mathbb R^{n})$,
%
\begin{multline}
\Vert f\Vert _{{\frac {n}{n-s}}}
\\
\,\,\,\le \alpha _{n,s} n \omega _{n}^{\frac{n+s}{n}}\Big(
\frac {1}{n}\int _{\snvtex}\Big( \int _{0}^{\infty }t^{-s}\, \Big\Vert
\frac{f(\cdot +t\xi )-f(\cdot )}{t}\Big\Vert _{1}\dvtex t\Big)^{-
\frac {n} {s}}\dvtex \xi \Big)^{- \frac {s} {n}}
\\
\le \alpha _{n,s}\int _{\mathbb R^{n}} \int _{\mathbb R^{n}}
\frac{\vert f(x)-f(y)\vert}{\vert x-y \vert ^{n+s}}\dvtex x\dvtex y.
\nonumber \end{multline}
There is equality in the first inequality if and only if $f$ is, up to
sets of measure zero, a constant multiple of the indicator function of
an ellipsoid. There is equality in the second inequality for radially symmetric
functions.
\end{theorem}

\noindent
In order to prove Theorem~\ref{thm_afracsobo}, we introduce the $s$-fractional
polar projection body $\ofpp{s} f$ associated to $f$, defined as the star
body whose gauge function (that is, the reciprocal of the radial function)
for $\xi \in \snvtex $ is
%
\begin{equation}
\Vert{\xi}\Vert _{\ofpp{s}f}^{s}= \int _{0}^{\infty }t^{-s}\, \int _{
\mathbb R^{n}}\Big\vert \frac{f(x +t\xi )-f(x)}{t}\Big\vert \dvtex x\dvtex t
\nonumber \end{equation}
(see Section~\ref{sec_fpp} for details). The affine fractional Sobolev
inequality can now be written as
%
\begin{equation}
\label{eq_fracsobvol}
\Vert f\Vert _{{\frac {n}{n-s}}} \le \alpha _{n,s} n \omega _{n}^{({n+s})/n}
\vol{\ofpp{s} f}^{- s/ n},
\end{equation}
where $\vol{\cdot}$ denotes $n$-dimensional Lebesgue measure. Since both
sides of~\eqref{eq_fracsobvol} are invariant under translations of
$f$, and for volume-preserving linear transformations
$\phi : \mathbb R^{n}\to \mathbb R^{n}$,
\begin{equation*}
\ofpp{s} \,(f\circ \phi ^{-1}) = \phi \ofpp{s} f,
\end{equation*}
it follows that~\eqref{eq_fracsobvol} is an affine inequality.

To show that Theorem~\ref{thm_afracsobo} implies the Sobolev--Zhang inequality,
we establish a result corresponding to~\eqref{eq_BBM}. For
$f\in W^{1,1}(\mathbb R^{n})$, we will prove that
%
\begin{equation}
\label{eq_limit}
\lim _{s\to 1^{-}} (1-s) \vol{\ofpp{s}f}^{- s/ n} = 2\, \vol{\opp f}^{-1/n},
\end{equation}
where $\opp f$ is the polar projection body of $f$ which was defined (with
different notation) in \cite{Zhang99} (also see \cite{LYZ2002b}) for
$\xi \in \snvtex $ by
%
\begin{equation}
\label{eq_polarprojection_functional}
\Vert \xi \Vert _{\opp f}= \frac {1}{2} \int _{\mathbb R^{n}} \vert
\langle \nabla f(x),\xi \rangle \vert \dvtex x.
\end{equation}
Moreover, we establish the corresponding statement for
$f\in BV(\mathbb R^{n})$ and also recover Wang's generalized Sobolev--Zhang
inequality (without characterization of the equality case) from Theorem~\ref{thm_afracsobo} (see Section~\ref{sec_limit1}).

We prove Theorem~\ref{thm_afracsobo} by first establishing affine fractional
P\'olya--Szeg\H o inequalities as a consequence of the Riesz rearrangement
inequality. For the equality case, we use anisotropic fractional Sobolev
norms (introduced in \cite{Ludwig:fracperi}) and the equality case of the
Riesz rearrangement inequality due to Burchard \cite{Burchard96}. For
$0<s<1$ and $E\subset \mathbb R^{n}$ of finite $s$-perimeter (see Section~\ref{sec_prelim}), we define the fractional polar projection body as
$\ofpp sE = \ofpp s\chi _{E}$, where $\chi _{E}$ is the indicator function
of $E$, and obtain from the affine fractional P\'olya--Szeg\H o inequalities
that
%
\begin{equation}
\label{eq_afracpetty}
\Big(\frac{\vol{\ofpp sE}}{\vol{\ofpp s\Bvtex}}\Big)^{-s/n} \geq \Big(
\frac{\vol{E}}{\vol{\Bvtex}}\Big)^{(n-s)/n}
\end{equation}
for $0<s<1$ and $E\subset \mathbb R^{n}$ of finite $s$-perimeter and finite
measure, where $\Bvtex $ is the Euclidean unit ball in $\mathbb R^{n}$. From
these fractional Petty projection inequalities, Theorem~\ref{thm_afracsobo} follows by a co-area formula for the fractional Sobolev
norm. In the limit as $s\to 1^{-}$, we obtain from~\eqref{eq_afracpetty} the generalized Petty projection inequality by Zhang
\cite{Zhang99} and Wang \cite{Tuo_Wang}. In addition, we prove versions
of our functional and geometric inequalities for Steiner symmetrization.

In Section~\ref{sec_GZ}, we connect fractional polar projection bodies
to radial mean bodies. For a convex body $E\subset \mathbb R^{n}$ (that
is, $E$ is a compact convex set with non-empty interior) and $p>-1$, Gardner
and Zhang \cite{GZ} defined the radial $p$-th mean body of $E$, by its
radial function for $\xi \in \snvtex $, as
%
\begin{equation}
\rho _{\omr p E}(\xi ) = \Big(\frac {1}{\vol{E}}\int _{E} \rho _{E-x}(
\xi )^{p} \dvtex x\Big)^{\frac{1}{p}},
\nonumber \end{equation}
where the limit of the above expression is taken for $p=0$. They proved
that
%
\begin{equation}
\label{eq_gardnerzhang}
\frac{\vol{\omr{p} \Svtex}}{\vol{\Svtex}} \le
\frac{\vol{\omr{p} E}}{\vol{ E}}
\end{equation}
for a convex body $E\subset \mathbb R^{n}$ and $-1< p<n $ with equality
precisely if $E$ is a simplex, and that the inequality is reversed for
$p>n$. Here, $\Svtex $ is any $n$-dimensional simplex in $\mathbb R^{n}$. We
show that
%
\begin{equation}
\ofpp{s} E =\Big(\frac {s}{2\vol{E}}\Big)^{\frac{1}{s}} \omr {-s} E
\nonumber \end{equation}
for every convex body $E\subset \mathbb R^{n}$ and $0<s<1$. Hence,~\eqref{eq_gardnerzhang} can be rewritten as the following reverse of inequality~\eqref{eq_afracpetty}:
%
\begin{equation}
\Big(\frac{\vol{\ofpp{s} E}}{\vol{\ofpp{s} \Svtex}}\Big)^{- s/ n} \le
\Big(\frac{\vol{ E}}{\vol{\Svtex}}\Big)^{ {(n-s)}/n}
\nonumber \end{equation}
for every convex body $E\subset \mathbb R^{n}$ and $0<s<1$. Moreover, using~\eqref{eq_afracpetty} and the Riesz rearrangement inequality, we show in
Theorem~\ref{thm_meanradial} that
%
\begin{equation}
\frac{\vol{\omr{p} E}}{\vol{ E}} \le \displaystyle
\frac{\vol{\omr{p} \Bvtex}}{\vol{\Bvtex}}
\nonumber \end{equation}
for a convex body $E\subset \mathbb R^{n}$ and $-1<p<n$ with equality precisely
if $E$ is an ellipsoid, and that the inequality is reversed for
$p>n$. Thereby, we establish a reverse of the Gardner--Zhang inequality~\eqref{eq_gardnerzhang} and prove a natural conjecture in the field.

\subsection*{Acknowledgments}

The authors thank the reviewers for their careful reading and their very
helpful remarks. J.~Haddad was supported by Grant RYC2021-031572-I, funded
by the Ministry of Science and Innovation/State Research Agency/10.13039/501100011033
and by the E.U. NextGeneration EU/Recovery, Transformation and Resilience
Plan, and partly by PAIDI2020, P20 00664 and by FAPERJ (JCNE 236508). M.~Ludwig
was supported, in part, by the Austrian Science Fund (FWF) 10.55776/P34446.

\section{Preliminaries}
\label{sec_prelim}

\subsection{Function spaces}

For $p\ge 1$ and measurable $f:\mathbb R^{n}\to \mathbb R$, let
\begin{equation*}
\Vert f\Vert _{p} =\Big(\int _{\mathbb R^{n}} \vert f(x)\vert ^{p}\dvtex x
\Big)^{1/p}.
\end{equation*}
We set $\{f\ge t\}=\{ x\in \mathbb R^{n}: f(x)\ge t\}$ for
$t\in \mathbb R$ and use corresponding notation for level sets, etc. We
say that $f$ is non-zero if $\{f\ne 0\}$ has positive measure, and we identify
functions that are equal up to a set of measure zero. For $p\ge 1$, let
\begin{equation*}
L^{p}(\mathbb R^{n}) = \Big\{f: \mathbb R^{n}\to \mathbb R: f
\text{ is measurable}, \,\Vert f\Vert _{p} < \infty \Big\}.
\end{equation*}
Here and below, when we use measurability and related notions, we refer
to the $n$-dimensional Lebesgue measure on $\mathbb R^{n}$.

For $E\subset \mathbb R^{n}$, the indicator function $\chi _{E}$ is defined
by $\chi _{E}(x)=1$ for $x\in E$ and $\chi _{E}(x)=0$ otherwise. We say
that two measurable sets are equivalent if their indicator functions are
equal up to a set of measure zero.

Let $0<s<1$. We define the $s$-fractional Sobolev space as
\begin{equation*}
W^{s,1}(\mathbb R^{n}) = \Big\{f \in L^{1}(\mathbb R^{n}) : \int _{
\mathbb R^{n}}\int _{\mathbb R^{n}} \frac{|f(x)-f(y)|}{|x-y|^{n+s}}
\dvtex x \dvtex y < \infty \Big\}.
\end{equation*}
For $f\in L^{1}(\mathbb R^{n})$, we say that $f$ admits a weak gradient
if there exists a vector field
$\nabla f:\mathbb R^{n} \to \mathbb R^{n}$ with measurable components such
that
\begin{equation*}
\int _{\mathbb R^{n}} \langle \nu (x), \nabla f(x) \rangle \dvtex x= -
\int _{\mathbb R^{n}} f(x)\div \nu (x) \, \dvtex x
\end{equation*}
for every $\nu \in C^{\infty}_{c}(\mathbb R^{n};\mathbb R^{n})$, where
$C_{c}^{\infty}(\mathbb R^{n};\mathbb R^{n})$ is the set of smooth vector
fields $\nu : \mathbb R^{n}\to \mathbb R^{n}$ with compact support and
$\div \nu $ denotes the divergence of $\nu $. We set
\begin{equation*}
W^{1,1}(\mathbb R^{n}) = \big\{f \in L^{1}(\mathbb R^{n}) : \vert
\nabla f \vert \in L^{1}(\mathbb R^{n})\big\}.
\end{equation*}
Let $\borel $ denote the class of Borel sets in $\mathbb R^{n}$. For
$f\in L^{1}(\mathbb R^{n})$, we say that $f$ is a function of bounded variation
on $\mathbb R^{n}$ if there is a finite vector-valued Radon measure
$\Dvtex f: \borel \to \mathbb R^{n}$ such that
%
\begin{equation}
\label{eq_bv}
\int _{\mathbb R^{n}} \langle \nu (x), \sigma _{f}(x) \rangle \dvtex
\vert \Dvtex f\vert (x) = - \int _{\mathbb R^{n}} f(x)\div \nu (x) \, \dvtex x
\end{equation}
for every $\nu \in C_{c}^{\infty}(\mathbb R^{n}; \mathbb R^{n})$, where
$\vert \Dvtex f\vert : \borel \to [0,\infty )$ denotes the variation measure
of $\Dvtex f$ and $\sigma _{f}: \mathbb R^{n} \to \snvtex $ the Radon--Nikodym
derivative of $\Dvtex f$ with respect to $\vert \Dvtex f\vert $. We write
$BV(\mathbb R^{n})$ for the space of $L^{1}$ functions of bounded variation.
For more information on functions of bounded variation, we refer to
\cite[Chapter~5]{Evans:Gariepy}.

\subsection{Symmetrization}

Let $E \subseteq \mathbb R^{n}$ be a Borel set of finite measure. The Schwarz
symmetral of $E$, denoted by $E^{\star}$, is the centered Euclidean ball
with the same volume as $E$.

For a non-negative integrable function $f$ and $t>0$, the superlevel set
$\{f \geq t\}$ is a measurable set of finite measure. The layer cake formula
states that
\begin{equation*}
f(x) = \int _{0}^{\infty }\chi _{\{f\geq t\}}(x) \dvtex t
\end{equation*}
for almost every $x \in \mathbb R^{n}$ and allows us to recover the function
from its superlevel sets. The Schwarz symmetral of $f$, denoted by
$f^{\star}$, is defined by
\begin{equation*}
f^{\star}(x) = \int _{0}^{\infty }\chi _{\{f\geq t\}^{\star}}(x) \dvtex t
\end{equation*}
for $x\in \mathbb R^{n}$. Hence, $f^{\star}$ is determined by the properties
of being radially symmetric and having superlevel sets that are balls of
the same measure as the superlevel sets of $f$. Note that
$f^{\star}$ is often called symmetric decreasing rearrangement of
$f$.

Let $\xi \in \snvtex $. The Steiner symmetral of $E$ in the direction
$\xi $, denoted by $E^{\xi}$, is defined by the property that for every
line $L$ parallel to $\xi $, the set $L \cap E^{\xi}$ is an interval of
the same length as $L \cap E$ and symmetric with respect to the subspace
orthogonal to $\xi $.

For a non-negative integrable function $f$, the Steiner symmetral
$f^{\xi}$ is defined by
\begin{equation*}
f^{\xi}(x) = \int _{0}^{\infty }\chi _{\{f\geq t\}^{\xi}}(x) \dvtex t
\end{equation*}
for $x\in \mathbb R^{n}$.

\subsection{Star bodies}

For a set $K \subseteq \mathbb R^{n}$ that is star-shaped (with respect
to the origin), the gauge function
$\|\cdot \|_{K} : \mathbb R^{n} \to [0,\infty ]$ is defined as
\begin{equation*}
\|x\|_{K} = \inf \{ \lambda > 0 : x \in \lambda K\},
\end{equation*}
and the radial function
$\rho _{K}:\mathbb R^{n} \setminus \{0\} \to [0,\infty ]$ as
\begin{equation*}
\rho _{K}(x) = \|x\|_{K}^{-1} = \sup \{\lambda \geq 0:\lambda x \in K
\}.
\end{equation*}
We call $K$ a star body if its radial function is strictly positive and
continuous in $\mathbb R^{n} \setminus \{0\}$.

We say that a star-shaped set is $p$-convex with $0< p \leq 1$ if
\begin{equation*}
\|x+y\|_{K}^{p} \leq \|x\|_{K}^{p} + \|y\|_{K}^{p}
\end{equation*}
for all $x,y \in \mathbb R^{n}$. A $p$-convex body is a star body that
is also $p$-convex. For more information on $p$-convex sets, see, for example,
\cite{KPR}.

For star bodies $K,L \subseteq \mathbb R^{n}$ and $p \in \mathbb R$, define
the dual mixed volume as
\begin{equation*}
\tilde V_{p}(K,L) = \frac {1}{n} \int _{\snvtex} \rho _{K}(\xi )^{n-p}
\rho _{L}(\xi )^{p} \dvtex \xi .
\end{equation*}
Note that
\begin{equation*}
\tilde V_{p}(K,K) = \vol{K}.
\end{equation*}
The dual mixed volume inequality (see
\cite[Section~9.3]{Schneider:CB2} or \cite[B.29]{Gardner}) states that
for $p<0$ or $p>n$,
%
\begin{equation}
\label{eq_mixedvolume}
\tilde V_{p}(K,L) \geq \vol{K}^{({n-p})/n} \vol{L}^{ p/n},
\end{equation}
and the reverse inequality holds for $0<p<n$. Equality holds for
$p \neq 0,n$ if and only if $K$ and $L$ are dilates, where we say that
star bodies $K$ and $L$ are dilates if $\rho _{K}=c\,\rho _{L}$ on
$\snvtex $ for some $c>0$.

For star bodies $K,L \subseteq \mathbb R^{n}$, define
%
\begin{equation}
\label{eq_dVlog}
\dVlog (K,L) = \frac {1}{n\vol{K}} \int _{\snvtex} \rho _{K}(\xi )^{n}
\log \Big(\frac{\rho _{L}(\xi )}{\rho _{K}(\xi )}\Big) \dvtex \xi .
\end{equation}
Note that
\begin{equation*}
\dVlog (K,L) = \lim _{p\to 0} \log \Big(
\frac{\tilde V_{p}(K,L)}{\vol{K}}\Big)^{1/p}.
\end{equation*}
By~\eqref{eq_mixedvolume}, this implies that
%
\begin{equation}
\label{eq_mixedvolume_log}
\dVlog (K,L) \leq \frac {1}{n}\log \Big(\frac{\vol{L}}{\vol{K}}\Big)
\end{equation}
for star bodies $K,L \subset \mathbb R^{n}$ (cf.\ \cite{GardnerHugWeilYe}).

\subsection{Convex bodies and the Petty projection inequality}

A set $K\subseteq \mathbb R^{n}$ is called a (proper) convex body if it
is compact and convex and has non-empty interior. For a convex body
$K$, the support function $h_{K}:\mathbb R^{n}\to \mathbb R$ is defined
as
\begin{equation*}
h_{K}(y)= \max \{\langle x,y\rangle : x\in K\}.
\end{equation*}
A convex body $K$ that contains the origin in its interior is also a star
body. The polar body of such $K$ is defined as
\begin{equation*}
K^{*}=\{y\in \mathbb R^{n}: \langle x,y\rangle \le 1 \text{ for all } x
\in K\}
\end{equation*}
and is a convex body containing the origin in its interior, while
\begin{equation*}
\Vert x\Vert _{K} = h_{K^{*}}(x)
\end{equation*}
for $x\in \mathbb R^{n}$.

For $K\subset \mathbb R^{n}$ a convex body, the projection body is defined
as the convex body whose support function for $\xi \in \snvtex $ is
%
\begin{equation}
h_{\opvtex K}(\xi )= \vol{\proj _{\xi ^{\perp}} (K)}_{n-1},
\nonumber \end{equation}
where $\proj _{\xi ^{\perp}}$ denotes the orthogonal projection to the
hyperplane, $\xi ^{\perp}$, orthogonal to $\xi $ and
$\vol{\cdot}_{n-1}$ stands for $(n-1)$-dimensional volume in
$\xi ^{\perp}$. Using the surface area measure $S(K, \cdot )$ of $K$, we
can also write
%
\begin{equation}
h_{\opvtex K}(\xi )= \frac {1}{2} \int _{\snvtex} \vert \langle \xi ,\eta
\rangle \vert \dvtex S(K,\eta )
\nonumber \end{equation}
for $\xi \in \snvtex $. Projection bodies were introduced by Minkowski, and
$h_{\opvtex K} (\xi )$ is called the brightness of $K$ in the direction
$\xi $ (see \cite{Gardner,Schneider:CB2} for the definition of surface
area measure and more information on projection bodies).

\eject

The important Petty projection inequality \cite{Petty} states that for
a convex body $K\subset \mathbb R^{n}$,
%
\begin{equation}
\label{eq_petty_convex}
\vol{ K}^{(n-1)/n}\le \frac{\omega _{n}}{\omega _{n-1}} \vol{\opp K}^{-
1/ n}
\end{equation}
with equality precisely for ellipsoids, where $\opp K$ is the polar body
of $\opvtex K$. A new proof using Theorem~\ref{thm_afracsobo} (and hence the
Riesz rearrangement inequality) is given in Section~\ref{sec_fracpetty}. A previous proof of~\eqref{eq_petty_convex} that also
uses the Riesz rearrangement inequality and the so-called convolution square
of a convex body was given by Schmuckenschl\"ager
\cite{Schmuckenschlager95}. For generalizations of projection bodies of
convex bodies, see, for example,
\cite{Boeroeczky2013,CLYZ2009,LRZ,Haberl:Schuster1,LYZ2000,LYZ2010b,Ludwig:Minkowski},
and for further results related to the Sobolev--Zhang inequality, see,
for example, \cite{HJM2016,LYZ2006,LYZ2002b,Nguyen16,Tuo_Wang:PSP}.

\subsection{Anisotropic fractional Sobolev norms and perimeters}

Let $0<s<1$ and let $K\subset \mathbb R^{n}$ be a star body. For
$f \in W^{s,1}(\mathbb R^{n})$, the anisotropic $s$-fractional Sobolev
norm of $f$ with respect to $K$ is
\begin{equation*}
\int _{\mathbb R^{n}}\int _{\mathbb R^{n}}
\frac{\vert f(x)-f(y)\vert}{\Vert x-y\Vert _{K}^{n+s}} \dvtex x\dvtex y.
\end{equation*}
It was defined in \cite{Ludwig:fracperi} for $K$ a convex body (also, see
\cite{Ludwig:fracnorm}). For a measurable set
$E \subseteq \mathbb R^{n}$, the anisotropic $s$-fractional perimeter
with respect to a convex body $K\subset \mathbb R^{n}$ was introduced in
\cite{Ludwig:fracperi} as
\begin{equation*}
P_{s}(E,K)= \int _{E} \int _{E^{c}}
\frac{1}{\Vert x-y\Vert _{K}^{n+s}} \dvtex x \dvtex y,
\end{equation*}
where $E^{c}= \mathbb R^{n} \backslash E$ is the complement of $E$ in
$\mathbb R^{n}$. The same definition will be used for a star body
$K\subset \mathbb R^{n}$. The anisotropic $s$-fractional perimeter of
$E$ with respect to $K$ is half of the anisotropic $s$-fractional Sobolev
norm of $\chi _{E}$ with respect to $K$. For a star body $K$ and measurable
$E\subset \mathbb R^{n}$, this implies that $P_{s}(E,K)<\infty $ if and
only if $\chi _{E}\in W^{s,1}(\mathbb R^{n})$. If $K$ is the Euclidean
unit ball, we write $P_{s}(E)$ and obtain the well-known Euclidean
$s$-fractional perimeter of $E$. We will use that
%
\begin{equation}
\label{eq_psb}
P_{s}(\Bvtex )= \frac{2^{1-s} \pi ^{(n-1)/2}n\omega _{n}}{s(n-s)}
\frac{\Gamma \left (\frac{1-s}{2}\right )}{\Gamma \left (\frac{n-s}{2}\right )}
=\frac{2^{1-s} n\omega _{n}\omega _{n-s}}{s(1-s)\omega _{1-s}}
\end{equation}
where $\omega _{q}=\pi ^{q/2}/\Gamma (q/2+1)$ for $q>0$ and
$\Gamma $ is the gamma function. This can be obtained from the results
in \cite{FFMMM} as described in \cite{Garofalo20}.

A simple way to obtain~\eqref{eq_psb} is to use the Blaschke--Petkantschin
formula \cite[Theorem 7.2.7]{SchneiderWeil}:
%
\begin{align*}
%
P_{s}&(\Bvtex )
\\
&= \int _{\Bvtex \cap L\ne \emptyset} \int _{\Bvtex \cap L} \int _{(\mathbb R^{n}
\backslash \Bvtex )\cap L}
\frac{\dvtex H^{1}(x)\dvtex H^{1}(y)}{\vert x-y\vert ^{1+s}} \dvtex L
\\
&=\frac{1}{2}\int _{\snvtex} \int _{\xi ^{\perp}} \int _{\Bvtex \cap (z+
\mathbb R\xi )}\int _{(\mathbb R^{n}\backslash \Bvtex )\cap (z+\mathbb R
\xi )} \frac{\dvtex H^{1}(x)\dvtex H^{1}(y) }{\vert x-y\vert ^{1+s}}\dvtex z \dvtex
\xi
\\
&=\frac{n\omega _{n} (n-1)\omega _{n-1}}{2} \int _{0}^{1} \int _{
\vert x\vert \le \sqrt{1-t^{2}}}\int _{\vert y\vert \le \sqrt{1-t^{2}}}
\frac{\dvtex H^{1}(x)\dvtex H^{1}(y)}{\vert x-y\vert ^{1+s}}\, t^{n-2}\dvtex t
\\
&=\frac{n\omega _{n} (n-1)\omega _{n-1}}{2} \int _{0}^{1}(1-t^{2})^{
\frac{1-s}{2}} t^{n-2}\dvtex t \int _{\vert x\vert \le 1}\int _{\vert y
\vert \le 1} \!\!\frac{\dvtex H^{1}(x)\dvtex H^{1}(y)}{\vert x-y\vert ^{1+s}}
\\
&=\frac{(n-1)n \omega _{n-1}\omega _{n}}{2^{s}s(1-s)}\, {\rm B}\biggl(\frac{1-s}{2}+1,\frac{n-1}{2}\biggr)
\end{align*}
where the outer integration in the first integral is on the affine Grassmannian
of lines in $\mathbb R^{n}$, and we use the $(n-1)$-dimensional Hausdorff
measure for the direction of lines and the $(n-1)$-dimensional Lebesgue
measure in the orthogonal complement of a line. Here $H^{1}$ denotes the
$1$-dimensional Hausdorff measure, and $\mathbb R\xi $ is the line with
direction $\xi $ while $\rm B$ is the beta function defined by
${\rm B}(a,b) = \int _{0}^{1} t^{a-1}(1-t)^{b-1} dt$. For the optimal constant
in~\eqref{eq_fracsobo}, we get
%
\begin{equation}
\label{eq_alpha}
\alpha _{n,s}= \frac{\omega _{n}^{(n-s)/n}}{2 P_{s}(\Bvtex )}
\end{equation}
since there is equality in~\eqref{eq_fracsobo} for indicator functions
of balls.

The following fractional co-area formula is a consequence of Fubini's theorem
(cf.\ \cite[Example~(2.9)]{Visintin90}):
%
\begin{align}
\,\,\int _{\mathbb R^{n}}\int _{\mathbb R^{n}} &
\frac{\vert f(x)-f(y)\vert}{\Vert x-y\Vert _{K}^{n+s}} \dvtex x\dvtex y
\nonumber \\
&=\int _{\mathbb R^{n}}\int _{\mathbb R^{n}} \int _{0}^{\infty }
\frac{\vert \chi _{\{f\ge t\}}(x)-\chi _{\{f\ge t\}}(y)\vert}{\Vert x-y\Vert _{K}^{n+s}}
\dvtex t \dvtex x\dvtex y
\label{eq_coarea}
\\
&= 2 \int _{0}^{\infty }P_{s}(\{ f\ge t\},K)\dvtex t
\nonumber \end{align}
for non-negative $f\in W^{s,1}(\mathbb R^{n})$ and a star body
$K\subset \mathbb R^{n}$.

\section{Fractional polar projection bodies}
\label{sec_fpp}

Let $0<s<1$ and $f\in BV(\mathbb R^{n})$. We discuss properties of the
$s$-fractional polar projection body $\ofpp{s} f$, that is, the star-shaped
set defined through its gauge function for $\xi \in \mathbb R^{n}$ as
%
\begin{equation}
\label{eq_defpp2}
\Vert{\xi}\Vert _{\ofpp{s}f}^{s}= \int _{0}^{\infty }t^{-s}\, \int _{
\mathbb R^{n}}\Big\vert \frac{f(x +t\xi )-f(x)}{t}\Big\vert \dvtex x\dvtex t.
\end{equation}
Notice that $\|\cdot \|_{\ofpp sf}$ is a $1$-homogeneous function.

\eject

We start with a simple observation. Let $K\subset \mathbb R^{n}$ be a star
body. Using~\eqref{eq_defpp2}, we obtain
%
\begin{align}
\int _{\mathbb R^{n}}\int _{\mathbb R^{n}} &
\frac{|f(x)-f(y)|}{\|x-y\|_{K}^{n+s}} \dvtex x \dvtex y
\nonumber \\
&= \int _{\mathbb R^{n}}\int _{\mathbb R^{n}}
\frac{|f(y+z)-f(y)|}{\|z\|_{K}^{n+s}} \dvtex z \dvtex y
\nonumber \\
&= \int _{\snvtex} \int _{0}^{\infty }{\|t \xi \|_{K}^{-n-s}} \int _{
\mathbb R^{n}} {|f(y+t\xi )-f(y)|}\,t^{n-1} \dvtex y \dvtex t \dvtex \xi
\label{eq_frac_func}
\\
&= \int _{\snvtex}{\| \xi \|_{K}^{-n-s}} \int _{0}^{\infty }t^{-s} \Big\|
\frac{f(\cdot +t\xi ) - f(\cdot )}{t}\Big\|_{1} \dvtex t \dvtex \xi
\nonumber \\
&= \int _{\snvtex} \rho _{K}(\xi )^{n+s} \rho _{\ofpp{s} f}(\xi )^{-s}
\dvtex \xi .
\nonumber \end{align}
We will use the following result to rewrite this in a more convenient way.

\begin{proposition}
\label{prop_Mconvexbody}
For non-zero $f \in W^{s,1}(\mathbb R^{n})$, the set $\,\ofpp{s} f$ is
an origin-symmetric $s$-convex body with the origin in its interior. Moreover,
there is $c>0$ depending only on $f$ such that
$\ofpp{s} f \subseteq c (1-s)^{1/s} \, \Bvtex $ for every $s \in (0,1)$.
\end{proposition}
\begin{proof}
First, note that since for $\xi \in \snvtex $ and $t>0$,
\begin{equation*}
\int _{\mathbb R^{n}}\vert {f(x -t\xi )-f(x)}\vert \dvtex x=\int _{
\mathbb R^{n}}\vert f(x)-f(x+t\xi )\vert \dvtex x,
\end{equation*}
the set $\ofpp{s} f$ is origin-symmetric.

Next, we show that $\rho _{\ofpp{s} f}$ is bounded by
$c (1-s)^{1/s} $ on $\snvtex $ for some $c>0$ depending only on $f$. We take
$r>1$ large enough so that
$\int _{r \Bvtex} \vert f(x)\vert \dvtex x \geq \frac {2}{3} \|f\|_{1}$. For given
$t \in (0,1)$, let $k>0$ be an integer such that
$2r \leq k t \leq 2r+t$. For $\xi \in \snvtex $, using changes of variables,
the triangle inequality, and the fact that $r\Bvtex $ is disjoint from
$r\Bvtex -kt\xi $, we obtain
%
\begin{align*}
\int _{\mathbb R^{n}}& |f(x+t\xi ) - f(x)| \dvtex x
\\
&= \frac {1}{k} \sum _{i=1}^{k} \int _{\mathbb R^{n}} |f(x+ i t\xi ) -
f(x+ (i-1)t \xi )| \dvtex x
\\
&\geq \frac {1}{k} \int _{\mathbb R^{n}} |f(x+ k t\xi ) - f(x)| \dvtex x
\displaybreak\\
&\geq \frac {1}{k} \Big( \int _{r \Bvtex} (|f(x)| - |f(x+kt\xi )|) \dvtex x
\\
&
\hskip 52pt
+ \int _{r \Bvtex -kt\xi}( |f(x+kt\xi )| - |f(x)|) \dvtex x \Big)
\\
&\geq \frac {t}{2r+t} \Big( 2 \int _{r \Bvtex} |f(x)| \dvtex x - 2\int _{(r
\Bvtex )^{c}} |f(x)| \dvtex x \Big)
\\
&\geq \frac {2t}{2r+1} \biggl( \frac {2}{3} \|f\|_{1} - \frac {1}{3} \|f\|_{1}
\biggr).
\end{align*}
Hence,
%
\begin{align}
\int _{0}^{\infty }t^{-s} \left \|\frac{f(\cdot +t\xi )-f(\cdot )}{t}
\right \|_{1} \dvtex t
&\geq \frac {2\|f\|_{1}}{3(2r+1)} \int _{0}^{1} t^{-s} \dvtex t
\nonumber \\
&\geq \frac {2\|f\|_{1}}{3(2r+1)} \frac {1}{1-s},
\nonumber
\end{align}
which implies that $\ofpp{s} f \subseteq c(1-s)^{1/s}\, \Bvtex $ for a constant
$c>0$ depending only on $f$.

Next, we show that $\Vert \cdot \Vert _{\ofpp{s}f}^{s}$ is sublinear on
$\mathbb R^{n}$. Indeed, for $\xi , \eta \in \mathbb R^{n}$, the triangle
inequality and a change of variables show that
%
\begin{equation}
\begin{split} \Vert \xi &+\eta \Vert _{\ofpp{s} f}^{s}
\\
&= \int \limits _{0}^{\infty }t^{-s-1} \|f(\cdot +t\xi +t\eta )-f(
\cdot )\|_{1} \dvtex t
\\
&\leq \int \limits _{0}^{\infty }t^{-s-1} ( \|f(\cdot +t\xi +t\eta )-f(
\cdot +t\xi )\|_{1} + \|f(\cdot +t\xi )-f(\cdot )\|_{1} ) \dvtex t
\label{eq_sublinearity}
\\
&= \|\xi \|_{\ofpp{s} f}^{s}+ \|\eta \|_{\ofpp{s} f}^{s},
\end{split}
\end{equation}
which shows that $\ofpp{s} f$ is an $s$-convex set.

Now, we show that $\ofpp{s} f$ has the origin in its interior. Using the
relation~\eqref{eq_frac_func} with $K=\Bvtex $, we get
%
\begin{equation}
\int _{\snvtex} \Vert \xi \Vert _{\ofpp{s}f}^{s} \dvtex \xi = \frac {1}{n}
\int _{\mathbb R^{n}}\int _{\mathbb R^{n}}
\frac{|f(x)-f(y)|}{|x-y|^{n+s}} \dvtex x \dvtex y,
\nonumber \end{equation}
which is finite since $f\in W^{s,1}(\mathbb R^{n})$. We choose $r>0$ large
enough so that the set
$A = \{\xi \in \snvtex : \Vert \xi \Vert _{\ofpp{s}f}^{s} < r\}$ has positive
$(n-1)$-dimensional Hausdorff measure and a basis
$\{\xi _{1}, \ldots , \xi _{n}\} \subseteq A$ of $\mathbb R^{n}$. Writing
$x \in \mathbb R^{n}$ as $x = \sum _{i=1}^{n} x_{i} \xi _{i}$ and using~\eqref{eq_sublinearity}, we get
%
\begin{equation}
\label{eq_bounded}
\begin{split} \|x\|_{\ofpp sf} &\leq \Big( \sum _{i=1}^{n} |x_{i}|^{s}
\|\xi _{i}\|_{\ofpp sf}^{s} \Big)^{1/s}
\\
&\leq r^{1/s} \Big( \sum _{i=1}^{n} |x_{i}|^{s} \Big)^{1/s}
\\
&\leq r^{1/s} n^{1/s-1/2}\, \Big( \sum _{i=1}^{n} |x_{i}|^{2} \Big)^{1/2}
\\
&\leq r^{1/s} n^{1/s-1/2} \|\Xi ^{-1}\| \, |x|,
\end{split}
\end{equation}
where $\|\Xi ^{-1}\|$ is the operator norm of the inverse of the matrix
with columns $\xi _{i}$, $i=1, \ldots , n$. This shows that
$\ofpp sf$ has the origin as interior point.

To see that the gauge function is continuous combine~\eqref{eq_sublinearity} and~\eqref{eq_bounded} to obtain
\begin{equation*}
\|\xi +\eta \|_{\ofpp{s} f}^{s} \leq \|\xi \|_{\ofpp{s} f}^{s} + \|
\eta \|_{\ofpp{s} f}^{s} \leq \|\xi \|_{\ofpp{s} f}^{s} + d\, |\eta |^{s},
\end{equation*}
where $d>0$ is independent of $\xi $ and $\eta $. Applying~\eqref{eq_sublinearity} to the vectors $\xi +\eta $ and $-\eta $, we get
\begin{equation*}
\|\xi \|_{\ofpp{s} f}^{s} \leq \|\xi +\eta \|_{\ofpp{s} f}^{s} + \|-
\eta \|_{\ofpp{s} f}^{s},
\end{equation*}
which again by~\eqref{eq_bounded} implies
\begin{equation*}
\|\xi +\eta \|_{\ofpp{s} f}^{s} \geq \|\xi \|_{\ofpp{s} f}^{s} - d\, |
\eta |^{s}.
\end{equation*}
This completes the proof.
\end{proof}

Let $f\in W^{s,1}(\mathbb R^{n})$. Using that $\ofpp{s} f$ is a star body,
we rewrite~\eqref{eq_frac_func} as
%
\begin{align}
\label{eq_dualmixed_func}
\int _{\mathbb R^{n}}\int _{\mathbb R^{n}}
\frac{|f(x)-f(y)|}{\|x-y\|_{K}^{n+s}} \dvtex x \dvtex y = n\, \tilde V_{-s}(K,
\ofpp{s} f).
\end{align}
We remark that~\eqref{eq_dualmixed_func} resembles Lutwak's formula (9.1)
from \cite{Lutwak:centroid},
%
\begin{equation}
\frac{n+1}{2}\vert K\vert \, V_{1}(L, \Gamma K)= \tilde V_{-1}(K,
\opp L)
\nonumber \end{equation}
for convex bodies $K, L\subset \mathbb R^{n}$, where $V_{1}$ is the first
mixed volume and $\Gamma K$ the centroid body of $K$.

\section[The Limit Case]{The limit case $s \to 1^{-}$ of fractional polar projection bodies}
\label{sec_limit1}

We prove~\eqref{eq_limit} for functions of bounded variation. For
$f \in BV(\mathbb R^{n})$, the polar projection body is defined for
$\xi \in \snvtex $ by
%
\begin{equation}
\Vert \xi \Vert _{\opp f} = \frac {1}{2}\int _{\mathbb R^{n}} |
\langle \sigma _{f}(x),\xi \rangle |\dvtex |\Dvtex f|(x)
\nonumber \end{equation}
(see \cite{Tuo_Wang}). Note that for $f\in W^{1,1}(\mathbb R^{n})$, this
definition coincides with~\eqref{eq_polarprojection_functional}.

\begin{theorem}
\label{thm_limitM}
Let $f \in BV (\mathbb R^{n})$. For $\xi \in \snvtex $,
\begin{equation*}
\lim _{s\to 1^{-}} (1-s)\Vert \xi \Vert _{\ofpp{s} f}^{s} = 2\,
\Vert \xi \Vert _{\opp f}.
\end{equation*}
Moreover,
\begin{equation*}
\lim _{s\to 1^{-}} (1-s) \vol{\ofpp{s} f}^{-s/n} = 2\, \vol{\opp f}^{-1/n}
\end{equation*}
and
%
\begin{equation}
\lim _{s\to 1^{-}} (1-s) \tilde V_{-s}(K, \ofpp{s} f)= 2\,\tilde V_{-1}(K,
\opp f)
\nonumber \end{equation}
for every star body $K\subset \mathbb R^{n}$.
\end{theorem}

\goodbreak
We require the following two lemmas.

\begin{lemma}
\label{lem_limit1D}
Let $\varphi \colon [0,\infty ) \to [0,\infty )$ be a measurable function.
If $\lim _{t\to 0^{+}} \varphi (t) = \varphi (0)$ and
$\int _{0}^{\infty }t^{-s_{0}} \varphi (t) \dvtex t < \infty $ for some
$s_{0} \in (0,1)$, then
\begin{equation*}
\lim _{s\to 1^{-}} (1-s) \int _{0}^{\infty }t^{-s} \varphi (t) \dvtex t =
\varphi (0).
\end{equation*}
\end{lemma}
\begin{proof}
Given $\varepsilon > 0$, we choose $\delta \in (0,1)$ so that
$0 \leq t \leq \delta $ implies
$|\varphi (t) - \varphi (0)| \leq \varepsilon $. We have
%
\begin{align}
\int _{0}^{\infty }t^{-s} &\varphi (t) \dvtex t
\nonumber \\
&= \int _{0}^{\delta }t^{-s} \varphi (0) \dvtex t + \int _{0}^{\delta }t^{-s}
(\varphi (t)-\varphi (0)) \dvtex t + \int _{\delta}^{\infty }t^{-s}
\varphi (t) \dvtex t
\nonumber \\
&= \frac{\delta ^{1-s}}{1-s} \varphi (0) + A + B
\nonumber \end{align}
where $(1-s)|A| \leq \varepsilon \delta ^{1-s} \leq \varepsilon $ and
\begin{equation*}
|B| \leq \delta ^{s_{0}-s} \int _{\delta}^{\infty }t^{-s_{0}}
\varphi (t) \dvtex t \leq \delta ^{s_{0}-s} \int _{0}^{\infty }t^{-s_{0}}
\varphi (t) \dvtex t
\end{equation*}
for $s \in (s_{0}, 1)$. Since $(1-s)B \to 0$ and
$\delta ^{1-s} \to 1$ as $s \to 1^{-}$, we can take $s$ sufficiently close
to $1$ so that $(1-s)|B| \leq \varepsilon $ and
$|\delta ^{1-s} \varphi (0) - \varphi (0)| \leq \varepsilon $, obtaining
$\left | (1-s) \int _{0}^{\infty }t^{-s} \varphi (t) \dvtex t - \varphi (0)
\right | \leq 3 \varepsilon $.
\end{proof}

\begin{lemma}
\label{lem_limitL1}
For $f \in BV(\mathbb R^{n})$ and $\xi \in \snvtex $,
\begin{equation*}
\lim _{t\to 0} \Big\|\frac{f(\cdot +t\xi )-f(\cdot )}{t}\Big\|_{1} =
\int _{\mathbb R^{n}} |\langle \sigma _{f}(x), \xi \rangle | \dvtex |\Dvtex f|(x).
\end{equation*}
\end{lemma}

\begin{proof}
Let $g:\mathbb R^{n}\to \mathbb R$ be a smooth function with compact support,
and write $\div _{x}$ for the divergence taken with respect to the variable~$x$.\vadjust{\eject} Using~\eqref{eq_bv}, we obtain
\begin{align*}
\int _{\mathbb R^{n}} \frac{f(x+t\xi )-f(x)}{t}\,&g(x) \dvtex x
\\
&= \int _{\mathbb R^{n}} f(x) \,\frac{g(x-t\xi )-g(x)}{t} \dvtex x
\\
&= - \int _{\mathbb R^{n}} f(x) \int _{0}^{1} \langle \nabla g(x-r t
\xi ), \xi \rangle \dvtex r \dvtex x
\\
&= - \int _{\mathbb R^{n}} f(x) \div _{x}\Big(\int _{0}^{1} g(x-r t
\xi ) \dvtex r\, \xi \Big)\dvtex x
\\
&= \int _{\mathbb R^{n}} \Big(\int _{0}^{1} g(x-r t\xi ) d r \Big)
\langle \sigma _{f}(x),\xi \rangle \dvtex |D f|(x).
\end{align*}
Therefore
%
\begin{equation}
\label{lem_limitL1:eq_weaklimit}
\begin{split} \lim _{t\to 0}\int _{\mathbb R^{n}}
\frac{f(x+t\xi )-f(x)}{t}&\,g(x) \dvtex x
\\
&= \int _{\mathbb R^{n}} g(x) \langle \sigma _{f}(x),\xi \rangle \dvtex |D
f|(x).
\end{split}
\end{equation}
\noindent
Let $\varepsilon > 0$. Since $|\Dvtex f|$ is a finite Radon measure, the set
of smooth functions with compact support is dense in the space of
$L^{1}$ functions w.r.t.\ $\vert \Dvtex f\vert $ (cf.~\cite[Proposition~7.9]{Folland}).
Hence, there is a smooth function $g:\mathbb R^{n} \to \mathbb R$ with
compact support such that
\begin{equation*}
\int _{\mathbb R^{n}} \vert g(x) - \sgn (\langle \sigma _{f}(x), \xi
\rangle )\vert \dvtex \vert \Dvtex f\vert (x) <\varepsilon
\end{equation*}
and $\|g\|_{\infty }\leq 1+\varepsilon $. By~\eqref{lem_limitL1:eq_weaklimit},
%
\begin{align}
(1+\varepsilon ) \liminf _{t\to 0} &\int _{\mathbb R^{n}}\Big\vert
\frac{f(x+t\xi )-f(x)}{t}\Big\vert \dvtex x
\nonumber \\
&\geq \lim _{t\to 0} \int _{\mathbb R^{n}} g(x)
\frac{f(x+t\xi )-f(x)}{t} \dvtex x
\nonumber \\
&= \int _{\mathbb R^{n}} g(x) \langle \sigma _{f}(x),\xi \rangle \dvtex |
\Dvtex f|(x)
\nonumber \\
&= \int _{\mathbb R^{n}} |\langle \sigma _{f}(x),\xi \rangle | \dvtex |
\Dvtex f|(x)
\nonumber \\
&
\phantom{=}
\, + \int _{\mathbb R^{n}} (g(x) - \sgn (\langle \sigma _{f}(x), \xi
\rangle )) \langle \sigma _{f}(x),\xi \rangle \dvtex |\Dvtex f|(x)
\nonumber \\
&\geq \int _{\mathbb R^{n}} |\langle \sigma _{f}(x),\xi \rangle | \dvtex |
\Dvtex f|(x) - \varepsilon .
\nonumber \end{align}
Since $\varepsilon >0$ was arbitrary, we obtain
\begin{equation*}
\liminf _{t\to 0} \Big\|\frac{f(\cdot +t\xi )-f(\cdot )}{t}\Big\|_{1}
\geq \int _{\mathbb R^{n}} |\langle \sigma _{f}(x), \xi \rangle | \dvtex |
\Dvtex f|(x).
\end{equation*}
For the opposite inequality, notice that
%
\begin{multline}
\left | \int _{\mathbb R^{n}} \Big(\int _{0}^{1} h(x-r t\xi ) \dvtex r
\Big) \langle \sigma _{f}(x),\xi \rangle \dvtex |\Dvtex f|(x) \right |
\\
\leq \|h\|_{\infty }\int _{\mathbb R^{n}} |\langle \sigma _{f}(x),
\xi \rangle |\dvtex |\Dvtex f|(x)
\nonumber \end{multline}
for every $h\in L^{\infty}(\mathbb R^{n})$.
\end{proof}

\begin{proof}[Proof of Theorem~\ref{thm_limitM}]
Define $\varphi : [0,\infty )\to [0,\infty )$ by
\begin{equation*}
\varphi (t) = \Big\|\frac{f(\cdot +t\xi )-f(\cdot )}{t}\Big\|_{1}
\end{equation*}
and note that $\varphi (t) \leq \frac{2\|f\|_{1}}{t}$ for $t> 0$. By Lemma~\ref{lem_limit1D} and Lemma~\ref{lem_limitL1},
\begin{equation*}
\lim _{s\to 1^{-}} (1-s) \int _{0}^{\infty }t^{-s} \Big\|
\frac{f(\cdot +t\xi )-f(\cdot )}{t}\Big\|_{1} \dvtex t = \int _{\mathbb R^{n}}
|\langle \sigma _{f}(x),\xi \rangle |\dvtex |\Dvtex f|(x).
\end{equation*}
By Proposition~\ref{prop_Mconvexbody}, we can use the dominated convergence
theorem to obtain,
\begin{align*}
\lim _{s\to 1^{-}}n\, &\vol{(1-s)^{-1/s} \ofpp{s} f}
\\
&= \lim _{s\to 1^{-}}\int _{\snvtex} \Big( (1-s) \int _{0}^{\infty }t^{-s}
\Big\|\frac{f(\cdot +t\xi )-f(\cdot )}{t}\Big\|_{1} \dvtex t \Big)^{-n/s}
\dvtex \xi
\\
&= \int _{\snvtex} \Big( \int _{\mathbb R^{n}} |\langle \sigma _{f}(x),
\xi \rangle |\dvtex \vert \Dvtex f\vert (x) \Big)^{-n} \dvtex \xi
\\
&= n\, \vol{ \tfrac {1}{2} \opp f}
\end{align*}
and
\begin{align*}
\lim _{s\to 1^{-}} n (1-s)\tilde V_{-s}(K, \ofpp{s} f) &= \lim _{s
\to 1^{-}}(1-s) \int _{\snvtex} \rho _{K}(\xi )^{n+s} \rho _{\ofpp{s} f}(
\xi )^{-s} \dvtex \xi
\\
&= 2 \int _{\snvtex} \rho _{K}(\xi )^{n+1} \rho _{\opp f}(\xi )^{-1} \dvtex
\xi
\\
&= 2n\,\tilde V_{-1}(K, \opp f),
\end{align*}
which completes the proof of the theorem.
\end{proof}

\section{Anisotropic fractional P\'olya--Szeg\H o inequalities}

Almgren and Lieb \cite{AlmgrenLieb} established the following P\'olya--Szeg\H o
inequality for fractional Sobolev norms. The equality case was settled
by Frank and Seiringer \cite{FrankSeiringer}. Let $0<s<1$ and
$n\ge 1$. For non-negative $f\in W^{s,1}(\mathbb R^{n})$,
%
\begin{equation}
\label{eq_EPS}
\int _{\mathbb R^{n}} \int _{\mathbb R^{n}}
\frac{\vert f(x)-f(y)\vert}{\vert x-y \vert ^{n+s}}\dvtex x\dvtex y \ge \int _{
\mathbb R^{n}} \int _{\mathbb R^{n}}
\frac{\vert f^{\star}(x)-f^{\star}(y)\vert}{\vert x-y \vert ^{n+s}}
\dvtex x\dvtex y,
\end{equation}
with equality precisely if $\{f \ge t\}$ is equivalent to a ball for almost
every $t>0$.

\goodbreak
The following improvement of~\eqref{eq_EPS} is key to proving the main
theorems in Section~\ref{sec_isoperimetric}. It is a variation of
\cite[Theorem 3.1]{andreas}.

\begin{theorem}%
\label{thm_fburchard}
Let $0<s<1$ and $K\subset \mathbb R^{n}$ a star body. For
$f :\mathbb R^{n} \to \mathbb R$ a non-negative integrable function,
%
\begin{equation}
\int _{\mathbb R^{n}} \int _{\mathbb R^{n}}
\frac{|f(x) - f(y)|}{\|x-y\|^{n+s}_{K}} \dvtex x \dvtex y \geq \int _{
\mathbb R^{n}} \int _{\mathbb R^{n}}
\frac{|f^{\star}(x) - f^{\star}(y)|}{\|x-y\|^{n+s}_{K^{\star}}} \dvtex x
\dvtex y.
\nonumber \end{equation}
For non-zero $f\in W^{s,1}(\mathbb R^{n})$, there is equality if and only
if $K$ is a centered ellipsoid and for almost every $t>0$, the level set
$\{f \geq t\}$ is equivalent to an ellipsoid homothetic to $K$ or has measure
zero.
\end{theorem}

\noindent
The inequality in Theorem~\ref{thm_fburchard} (without the equality case)
follows rather directly from the Riesz rearrangement inequality, which
is stated in full generality, for example, in \cite{BLL}.

\begin{theorem}[Riesz's rearrangement inequality]
\label{thm_BLL}
For $f,g,k:\mathbb R^{n} \to \mathbb R$ non-negative measurable functions,
\begin{equation*}
\int _{\mathbb R^{n}}\int _{\mathbb R^{n}} f(x) k(x-y) g(y) \dvtex x \dvtex y
\leq \int _{\mathbb R^{n}}\int _{\mathbb R^{n}} f^{\star}(x) k^{\star}(x-y)
g^{\star}(y) \dvtex x \dvtex y,
\end{equation*}
and the right-hand side is infinite if the left-hand side is infinite.
\end{theorem}

To establish the equality case in Theorem~\ref{thm_fburchard}, we use the
characterization of equality cases in a version of the Riesz rearrangement
inequality due to Burchard \cite{Burchard96}.

\begin{theorem}[Burchard]
\label{thm_burchard}
Let $A,B$ and $C$ be sets of finite positive measure in
$\,\mathbb R^{n}$ and denote by $\alpha , \beta $ and $\gamma $ the radii
of their Schwarz symmetrals $A^{\star}, B^{\star}$ and $C^{\star}$. For
$\,\vert \alpha - \beta \vert < \gamma < \alpha +\beta $, there is equality
in
%
\begin{multline}
\int _{\mathbb R^{n}}\int _{\mathbb R^{n}} \chi _{A}(y) \chi _{B}(x-y)
\chi _{C}(x) \dvtex x \dvtex y
\\
\leq \int _{\mathbb R^{n}}\int _{\mathbb R^{n}} \chi _{A^{\star}}(y)
\chi _{B^{\star}}(x-y) \chi _{C^{\star}}(x) \dvtex x \dvtex y
\nonumber \end{multline}
if and only if, up to sets of measure zero,
\begin{equation*}
A=a+\alpha D, B = b+\beta D, C = c+\gamma D,
\end{equation*}
where $D$ is a centered ellipsoid, and $a,b$ and $c=a+b$ are vectors in
$\mathbb R^{n}$.
\end{theorem}

By the co-area formula~\eqref{eq_coarea}, Theorem~\ref{thm_fburchard} is
an immediate consequence of the following result for anisotropic fractional
perimeters. We say that $A\subset \mathbb R^{n}$ is homothetic to
$D\subset \mathbb R^{n}$ if there are $a\in \mathbb R^{n}$ and
$\alpha >0$ such that $A=a+\alpha \,D$.

\begin{theorem}%
\label{thm_perburchard}
Let $0<s<1$ and $K\subset \mathbb R^{n}$ be a star body. For
$E\subset \mathbb R^{n}$ measurable,
%
\begin{equation}
\label{eq_perburchard}
\int _{E}\int _{E^{c}}\frac{1}{\|x-y\|^{n+s}_{K}} \dvtex x \dvtex y \geq
\int _{E^{\star}}\int _{(E^{\star})^{c}}
\frac{1}{\|x-y\|^{n+s}_{K^{\star}}} \dvtex x \dvtex y.
\end{equation}
For $E$ with finite positive $s$-perimeter, equality holds if and only
if $K$ is a centered ellipsoid and $E$ is equivalent to an ellipsoid homothetic
to~$K$.
\end{theorem}

\begin{proof}
For $z\ne 0$,
\begin{equation*}
\|z\|_{K}^{-n-s} = \int _{0}^{\infty }k_{t}(z) \dvtex t \text{, where } k_{t}(z)
= \chi _{t^{- 1/({n+s})}K}(z),
\end{equation*}
and using Fubini's theorem, we obtain
\begin{equation*}
\int _{E}\int _{E^{c}} \frac{1}{\|x-y\|^{n+s}_{K}} \dvtex x \dvtex y = \int _{0}^{
\infty }\int _{\mathbb R^{n}}\int _{\mathbb R^{n}} \chi _{E}(x) \chi _{E^{c}}(y)
k_{t}(x-y) \dvtex x \dvtex y \dvtex t.
\end{equation*}
For fixed $t \in (0,\infty )$, we have
%
\begin{align}
\int _{\mathbb R^{n}}\int _{\mathbb R^{n}} \chi _{E}(x) &\chi _{E^{c}}(y)
k_{t}(x-y) \dvtex x \dvtex y
\nonumber \\
&= \int _{\mathbb R^{n}}\int _{\mathbb R^{n}} \big( \chi _{E}(x)-
\chi _{E}(x) \chi _{E}(y) \big)k_{t}(x-y) \dvtex x \dvtex y
\label{eq_rewrite}
\\
&= t^{-\frac {n}{n+s}} \vol{K} \vol{E} - \int _{\mathbb R^{n}}\int _{
\mathbb R^{n}}\chi _{E}(x)\, k_{t}(x-y) \chi _{E}(y) \dvtex x \dvtex y.
\nonumber \end{align}
Clearly, the first term is invariant under Schwarz symmetrization. For
the second term, we can apply Theorem~\ref{thm_BLL} to show that the left-hand
side does not increase under Schwarz symmetrization. Thus
%
\begin{multline}
\int _{\mathbb R^{n}}\int _{\mathbb R^{n}} \chi _{E}(x) k_{t}(x-y)
\chi _{E^{c}}(y) \dvtex x \dvtex y
\\
\ge \int _{\mathbb R^{n}}\int _{\mathbb R^{n}} \chi _{E^{\star}}(x)\, k^{
\star}_{t}(x-y) \chi _{(E^{\star})^{c}}(y) \dvtex x \dvtex y
\nonumber \end{multline}
for $t>0$ and integrating this inequality for $t\in (0,\infty )$, we obtain~\eqref{eq_perburchard}.

If there is equality in~\eqref{eq_perburchard}, it follows from~\eqref{eq_rewrite} that for almost every $t>0$,
\begin{equation*}
\int _{\mathbb R^{n}}\int _{\mathbb R^{n}} \chi _{E}(x) \chi _{t^{-1/(n+s)}K}(x-y)
\chi _{E}(y) \dvtex x \dvtex y
\end{equation*}
is invariant under Schwarz symmetrization. Now we may apply Theorem~\ref{thm_burchard} with $A = C = E$ and $B = t^{- 1/({n+s})} K$ for suitable
$t>0$. Since $A=C$, we have $a=c$, and this implies $b=0$. Hence,
$K$ is a centered ellipsoid, and $E$ is equivalent to an ellipsoid homothetic
to $K$.
\end{proof}

To establish corresponding inequalities for Steiner symmetrization, we
require the following result.

\begin{theorem}[Rogers \cite{Rogers56}]
\label{thm_sBLL}
Let $\xi \in \snvtex $. For $f,g,k:\mathbb R^{n} \to \mathbb R$ non-negative
and measurable,
\begin{equation*}
\int _{\mathbb R^{n}}\int _{\mathbb R^{n}} f(x) k(x-y) g(y) \dvtex x \dvtex y
\leq \int _{\mathbb R^{n}}\int _{\mathbb R^{n}} f^{\xi}(x) k^{\xi}(x-y)
g^{\xi}(y) \dvtex x \dvtex y.
\end{equation*}
\end{theorem}

\noindent
We remark that the above theorem is a special case of the inequalities
of Rogers~\cite{Rogers57} and Brascamp--Lieb--Luttinger
\cite[Lemma~3.2]{BLL}.

\goodbreak
As in the case of Schwarz symmetrization, we obtain the following consequences.

\begin{theorem}%
\label{thm_perburchard_steiner}
Let $0<s<1$ and $\xi \in \snvtex $. If $K\subset \mathbb R^{n}$ is a star body,
then
%
\begin{equation}
\int _{E}\int _{E^{c}}\frac{1}{\|x-y\|^{n+s}_{K}} \dvtex x \dvtex y \geq
\int _{E^{\xi}}\int _{(E^{\xi})^{c}}
\frac{1}{\|x-y\|^{n+s}_{K^{\xi}}} \dvtex x \dvtex y
\nonumber \end{equation}
for $E\subset \mathbb R^{n}$ measurable.
\end{theorem}

\begin{theorem}%
\label{thm_fburchard_steiner}
Let $0<s<1$ and $\xi \in \snvtex $. If $K\subset \mathbb R^{n}$ is a star body,
then
%
\begin{equation}
\int _{\mathbb R^{n}} \int _{\mathbb R^{n}}
\frac{|f(x) - f(y)|}{\|x-y\|^{n+s}_{K}} \dvtex x \dvtex y \geq \int _{
\mathbb R^{n}} \int _{\mathbb R^{n}}
\frac{|f^{\xi}(x) - f^{\xi}(y)|}{\|x-y\|^{n+s}_{K^{\xi}}} \dvtex x \dvtex y
\nonumber \end{equation}
for $f :\mathbb R^{n} \to \mathbb R$ non-negative and integrable.
\end{theorem}

\section{Affine fractional P\'olya--Szeg\H o inequalities}
\label{sec_fracpettyproj}

We establish the following affine version of the fractional P\'olya--Szeg\H o
inequality~\eqref{eq_EPS} by Almgren and Lieb.

\begin{theorem}%
\label{thm_aPS}
For $0<s<1$ and non-negative $f\in W^{s,1}(\mathbb R^{n})$,
%
\begin{equation}
\label{eq_aPS}
\vol{\ofpp{s} f}^{-s/n} \geq \vol{\ofpp{s} f^{\star}}^{-s/n}.
\end{equation}
There is equality if and only if for almost every $t>0$, the level set
$\{f \ge t\}$ has measure zero or is homothetic to an ellipsoid (independent
of $t$) up to a set of measure zero.
\end{theorem}

\begin{proof}
Let $K\subset \mathbb R^{n}$ be a star body. By Theorem~\ref{thm_fburchard},~\eqref{eq_dualmixed_func}, and~\eqref{eq_mixedvolume}, we have
%
\begin{align}
\tilde V_{-s} (K, \ofpp sf) &\geq \tilde V_{-s} (K^{\star}, \ofpp sf^{
\star})
\nonumber \\
&\geq \vol{K^{\star}}^{({n+s})/{n}} \vol{\ofpp sf^{\star}}^{- s/n}
\nonumber \\
&= \vol{K}^{({n+s})/{n}} \vol{\ofpp sf^{\star}}^{- s/n}.
\nonumber \end{align}
Setting $K = \ofpp sf$, we obtain that
%
\begin{align}
\vol{\ofpp sf} = \tilde V_{-s} (\ofpp sf, \ofpp sf) \geq
\vol{\ofpp sf}^{({n+s})/{n}} \vol{\ofpp sf^{\star}}^{-\frac {s}{n}},
\nonumber \end{align}
which completes the proof of the inequality. The equality case follows
from the equality case of Theorem~\ref{thm_fburchard}.
\end{proof}

Multiplying~\eqref{eq_aPS} by $(1-s)$ and letting $s\to 1^{-}$, we obtain
from Theorem~\ref{thm_aPS} and Theorem~\ref{thm_limitM} the following affine
P\'olya--Szeg\H o inequality by Cianchi, Lutwak, Yang, and Zhang
\cite{CLYZ2009} (see also \cite{Nguyen16}):
%
\begin{equation}
\vol{\opp f}^{-1/n} \geq \vol{\opp f^{\star}}^{-1/n}
\nonumber \end{equation}
for $f\in W^{1,1}(\mathbb R^{n})$.

Using the same proof as for Theorem~\ref{thm_aPS} for Steiner symmetrization
and replacing Theorem~\ref{thm_fburchard} with Theorem~\ref{thm_fburchard_steiner}, we obtain the following result.

\begin{theorem}%
\label{thm_aSteiner}
Let $0<s<1$ and $\xi \in \snvtex $. Then
%
\begin{equation}
\vol{\ofpp{s} f}^{-s/n} \geq \vol{\ofpp{s} f^{\xi}}^{-s/n}
\nonumber \end{equation}
for non-negative $f\in W^{s,1}(\mathbb R^{n})$.
\end{theorem}

As before, we obtain from Theorem~\ref{thm_aSteiner} and Theorem~\ref{thm_limitM} the following affine inequality:
%
\begin{equation}
\vol{\opp f}^{-1/n} \geq \vol{\opp f^{\xi}}^{-1/n}
\nonumber \end{equation}
for $f\in W^{1,1}(\mathbb R^{n})$ and $\xi \in \snvtex $.

\section{Fractional Petty projection inequalities}
\label{sec_fracpetty}

A set $E\subset \mathbb R^{n}$ of finite measure is a set of finite perimeter
if its indicator function $\chi _{E}$ is in $BV(\mathbb R^{n})$. This allows
us to translate results for functions in $BV(\mathbb R^{n})$ to sets of
finite perimeter and finite measure. The following result is an immediate
consequence of the affine fractional P\'olya--Szeg\H o inequality from
Theorem~\ref{thm_aPS}.

\begin{theorem}%
\label{thm_fracpetty1}
For $\,0<s<1$ and $E\subset \mathbb R^{n}$ of finite $s$-perimeter and
finite measure,
\begin{equation*}
\vol{\ofpp{s} E}^{-s/n} \geq \vol{\ofpp{s} E^{\star}}^{-s/n}
\end{equation*}
with equality precisely if $E$ is equivalent to an ellipsoid.
\end{theorem}

\noindent
Here, we write $\ofpp{s} E$ for the fractional polar projection body of
$\chi _{E}$. Note that
\begin{equation*}
\Vert \xi \Vert ^{s}_{\ofpp s E}= \int _{0}^{\infty }t^{-s}
\frac{\vol{E\triangle (E-t\xi )}}{t}\dvtex t
\end{equation*}
for $E\subset \mathbb R^{n}$ of finite $s$-perimeter and finite measure.
Also note that~\eqref{eq_dualmixed_func} implies that
%
\begin{equation}
\label{eq_fracperi_proj}
2 \int _{E}\int _{E^{c}} \frac{1}{\|x-y\|_{K}^{n+s}} \dvtex x \dvtex y = n\,
\tilde V_{-s}(K, \ofpp{s} E)
\end{equation}
for $K\subset \mathbb R^{n}$ a star body.

From Theorem~\ref{thm_fracpetty1}, we easily obtain the first inequality
of the following result, which we call the fractional Petty projection
inequality.
%
\begin{theorem}%
\label{thm_fracpetty}
For $\,0<s<1$ and $E\subset \mathbb R^{n}$ of finite $s$-perimeter and
finite measure,
%
\begin{equation}
\label{eq_fracpetty}
\Big(\frac{\vol{ E}}{\vol{\Bvtex}}\Big)^{ {(n-s)}/n} \le \Big(
\frac{\vol{\ofpp{s} E}}{\vol{\ofpp{s} \Bvtex}}\Big)^{- s/ n}
\le \frac{ P_{s}(E)}{P_{s}(\Bvtex )}.
\end{equation}
There is equality in the first inequality if and only if $E$ is equivalent
to an ellipsoid. There is equality in the second inequality if and only
if $E$ is equivalent to a set of constant $s$-fractional brightness.
\end{theorem}

\noindent
Here, we say that $E\subset \mathbb R^{n}$ is of constant $s$-fractional
brightness if $\ofpp{s} E$ is a dilate of $\Bvtex $. The second inequality
in~\eqref{eq_fracpetty} and its equality case follow directly from~\eqref{eq_fracperi_proj} and the dual mixed volume inequality~\eqref{eq_mixedvolume}.

Let $E\subset \mathbb R^{n}$ be a set of finite perimeter and finite measure.
The perimeter of $E$ is defined as $P(E)= \vert \Dvtex \chi _{E}\vert $. By
D\'avila's version of~\eqref{eq_BBM},
\begin{equation*}
\lim _{s\to 1^{-}} \frac{P_{s}(E)}{P_{s}(\Bvtex )} = \frac{P(E)}{P(\Bvtex )}.
\end{equation*}
Taking the limit $s\to 1^{-}$ in~\eqref{eq_fracpetty}, we obtain from this
and Theorem~\ref{thm_limitM} that
%
\begin{equation}
\label{eq_petty}
\Big(\frac{\vol{ E}}{\vol{\Bvtex}}\Big)^{ {(n-1)}/n} \le \Big(
\frac{\vol{\opp E}}{\vol{\opp \Bvtex}}\Big)^{- 1/ n}
\le \frac{ P(E)}{P(\Bvtex )}
\end{equation}
for $E\subset \mathbb R^{n}$ of finite perimeter and finite measure. The
first inequality is the generalized Petty projection inequality, which
was proved by Gaoyong Zhang~\cite{Zhang99} for compact sets with piecewise
$C^{1}$ boundary and in full generality by Tuo Wang \cite{Tuo_Wang}. For
$E$ a convex body, it is the classical Petty projection inequality~\eqref{eq_petty_convex}. There is equality in the second inequality of~\eqref{eq_petty} if and only if $E$ is of constant brightness.

The Steiner inequality for fractional polar projection bodies is contained
in the following result, which is an immediate consequence of Theorem~\ref{thm_aSteiner}.

\begin{theorem}%
\label{thm_asSteiner}
Let $\,0<s<1$ and $\xi \in \snvtex $. Then
%
\begin{equation}
\vol{\ofpp{s} E}^{-s/n} \geq \vol{\ofpp{s} E^{\xi}}^{-s/n}
\nonumber \end{equation}
for $E\subset \mathbb R^{n}$ of finite $s$-perimeter and finite Lebesgue
measure.
\end{theorem}

\noindent
Taking the limit as $s\to 1^{-}$ and using Theorem~\ref{thm_limitM}, we
obtain from Theorem~\ref{thm_asSteiner} the following Steiner inequality
for polar projection bodies, which was recently established by Youjiang
Lin \cite{Lin21},
\begin{equation*}
\vol{\opp E}^{-1/n} \geq \vol{\opp E^{\xi}}^{-1/n}
\end{equation*}
for $E\subset \mathbb R^{n}$ a set of finite perimeter and finite Lebesgue
measure and $\xi \in \snvtex $. For convex bodies, this inequality was established
by Lutwak, Yang, and Zhang \cite{LYZ2000} and for compact sets, by Wang
and Xiao \cite{WangXiao}.

\section[Proof of Theorem 1]{Proof of Theorem~\ref{thm_afracsobo}}
\label{sec_isoperimetric}

Before proving Theorem~\ref{thm_afracsobo}, we establish an inequality
involving the Lorentz quasi-norm, which might be of independent interest.
For $1\leq q < \infty $ and $1 \leq r < \infty $, the Lorentz quasi-norm
is defined by
\begin{equation*}
\|f\|_{q,r} = \left ( \int _{0}^{\infty }t^{r-1} \big\vert \{|f| \ge t
\}\big\vert ^{r/q} \dvtex t \right )^{1/r}
\end{equation*}
for measurable functions $f$ on $\mathbb R^{n}$.

\begin{proposition}
\label{thm_afracsobolorenz}
For $\,0<s<1$ and $f\in W^{s,1}(\mathbb R^{n})$,
%
\begin{equation}
\Vert f\Vert _{\frac{n}{n-s},1} \le \alpha _{n,s} n \omega _{n}^{({n+s})/n}
\vol{\ofpp{s} f}^{- s/ n}
\nonumber \end{equation}
There is equality if and only if $f$ has constant sign and for almost every
$t>0$, the level set $\{\vert f\vert \geq t\}$ has measure zero or is homothetic
to an ellipsoid (independent of $t$) up to a set of measure zero.
\end{proposition}
\begin{proof}
First, assume that $f$ is non-negative. By the co-area formula~\eqref{eq_coarea} and~\eqref{eq_dualmixed_func},
%
\begin{align}
\int _{\mathbb R^{n}}\int _{\mathbb R^{n}}
\frac{\vert f(x)-f(y)\vert}{\Vert x-y\Vert _{K}^{n+s}}\dvtex x\dvtex y &= 2
\int _{0}^{\infty }P_{s}(\{f \ge t\}, K)\dvtex t
\nonumber \\
&= n\int _{0}^{\infty }\tilde V_{-s}(K, \ofpp{s}\,\{f \ge t\})\dvtex t.
\nonumber \end{align}
Hence, setting $K= \ofpp{s}f$, we obtain from~\eqref{eq_dualmixed_func} and the dual mixed volume inequality~\eqref{eq_mixedvolume} that
%
\begin{align}
\vol{\ofpp{s} f} &= \int _{0}^{\infty }\tilde V_{-s}(\ofpp{s} f,
\ofpp{s} \,\{f \ge t\}) \dvtex t
\nonumber \\
&\ge \vol{\ofpp{s} f}^{({n+s})/n} \int _{0}^{\infty }
\vol{\ofpp{s} \,\{f \ge t\}}^{-s/n}\dvtex t
\nonumber \end{align}
and
%
\begin{equation}
\vol{\ofpp{s} f}^{-s/n}\ge \int _{0}^{\infty }
\vol{ \ofpp{s} \,\{f \ge t\}}^{-s/n} \dvtex t.
\nonumber \end{equation}
Since there is equality in~\eqref{eq_fracsobo} for indicator functions
of balls, the equality case of the dual mixed volume inequality, applied
to~\eqref{eq_fracperi_proj} for $K=\Bvtex $, shows that
\begin{equation*}
{\alpha _{n,s} n \omega _{n}^{({n+s})/n}} =\vol{\ofpp s\Bvtex}^{s/n}{
\vol{\Bvtex}^{({n-s})/n}}.
\end{equation*}
By the fractional Petty projection inequality from Theorem~\ref{thm_fracpetty}, it now follows that
%
\begin{equation}
\label{eq_frcpetty}
\begin{split} \vol{\ofpp{s} f}^{- s/ n} &\ge
\frac{\vol{\ofpp s\Bvtex}^{-s/n}}{\vol{\Bvtex}^{({n-s})/n}} \int _{0}^{
\infty }\vert \{f \ge t\}\vert ^{(n-s)/n}\dvtex t
\\
&= \frac {1} {\alpha _{n,s} n \omega _{n}^{({n+s})/n}} \|f\|_{
\frac{n}{n-s},1}.
\end{split}
\end{equation}
In particular, the last term is finite. The equality case follows from
the equality case of Theorem~\ref{thm_fracpetty}.

For general $f$ and $x,y\in \mathbb R^{n}$, we use
$\vert f(x) - f(y) \vert \geq \vert \vert f(x) \vert - \vert f(y)
\vert \vert $, where equality holds if and only if $f(x)$ and $f(y)$ are
both non-negative or non-positive. Applying this inequality in the definition
of $\ofpp sf$, we obtain
\begin{equation*}
\vol{\ofpp s{|f|}}^{-s/n} \leq \vol{\ofpp s{f}}^{-s/n}
\end{equation*}
with equality if and only if $f$ has constant sign for almost everywhere
on $\mathbb R^{n}$. Using~\eqref{eq_frcpetty} for $\vert f\vert $, we obtain
the inequality of the theorem and its equality case.
\end{proof}

\goodbreak
We include the proof of the following well-known lemma for completeness
(cf.\ \cite{Zhang99} for $s=1$).

\begin{lemma}%
\label{lem_zhang_s}
Let $0<s<1$. If $g:\mathbb R^{n}\to [0,\infty )$ is measurable, then
%
\begin{align}
\label{eq_zhang_s}
\Big(\int _{\mathbb R^{n}} g(x)^{n/({n-s})}\dvtex x\Big)^{({n-s})/n} &
\le \int _{0}^{\infty }\vol{\{g \ge t\}}^{({n-s})/n}\dvtex t.
\end{align}
If the right-hand side is finite, then there is equality precisely if
$g=c\chi _{E}$ for some $E\subset \mathbb R^{n}$ of finite measure and
$c\ge 0$.
\end{lemma}

\begin{proof}
Let the right-hand side of~\eqref{eq_zhang_s} be finite. By Fubini's theorem,
we have
%
\begin{align}
\int _{\mathbb R^{n}} g(x)^{n/({n-s})}\dvtex x &= \frac{n}{n-s} \int _{0}^{
\infty }t^{ s/( {n-s})} \vol{\{g \ge t\}} \dvtex t.
\nonumber \end{align}
Since $r\mapsto \vol{\{g \ge r\}}$ is monotone decreasing, we obtain for
$t>0$,
%
\begin{align}
t^{ s/({n-s})} \vol{\{g \ge t\}} &=\big(t\, \vol{\{g \ge t\}}^{( {n-s})/n}
\big)^{ s/({n-s})} \vol{\{g \ge t\}}^{({n-s})/n}
\nonumber \\
&\le \bigg(\int _{0}^{t} \vol{\{g \ge r\}}^{({n-s})/n}\dvtex r\bigg)^{ s/({n-s})} \vol{\{g \ge t\}}^{({n-s})/n}
\nonumber \\
&=\frac{n-s}{n} \frac{\dvtex \,\, }{\dvtex t} \bigg(\int _{0}^{t} \vol{\{g \ge r\}}^{({n-s})/n}\dvtex r\bigg)^{ n/({n-s})}.
\nonumber \end{align}
Hence
%
\begin{align}
\Big(\int _{\mathbb R^{n}} g(x)^{ n/({n-s})}\dvtex x\Big)^{({n-s})/n} &
\le \int _{0}^{\infty }\vol{\{g \ge t\}}^{({n-s})/n}\dvtex t
\nonumber \end{align}
and there is equality precisely if $g=c\chi _{E}$ for some
$E\subset \mathbb R^{n}$ of finite measure and $c\ge 0$.
\end{proof}

\begin{proof}[Proof of Theorem~\ref{thm_afracsobo}]
The first inequality follows from Proposition~\ref{thm_afracsobolorenz} and Lemma~\ref{lem_zhang_s} applied to
$|f|$. The equality case is a consequence of the equality cases of these
inequalities.

\eject

For the second inequality, we set $K=\Bvtex $ in~\eqref{eq_dualmixed_func} and apply the dual mixed volume inequality~\eqref{eq_mixedvolume} to obtain
%
\begin{align}
\int _{\mathbb R^{n}}\int _{\mathbb R^{n}}
\frac{\vert f(x)-f(y)\vert}{\vert x-y\vert ^{n+s}}\dvtex x\dvtex y &= n
\tilde V_{-s}(\Bvtex , \ofpp sf)
\nonumber \\
&\geq n \omega _{n}^{( {n+s})/n} \vol{\ofpp sf}^{-s/n}.
\nonumber \end{align}
There is equality precisely if $\ofpp sf$ is a ball, which is the case
for radially symmetric functions.
\end{proof}

We remark that we have shown that there is equality in the second inequality
of Theorem~\ref{thm_afracsobo} precisely if $\ofpp sf$ is a ball.

\goodbreak

\section{Fractional Zhang projection inequalities and radial mean bodies}
\label{sec_GZ}

Let $E\subset \mathbb R^{n}$ be a convex body. For $p>-1$, Gardner and
Zhang \cite{GZ} defined the radial $p$-th mean body of $E$, by its radial
function for $\xi \in \snvtex $, as
\begin{equation*}
\rho _{\omr p E}(\xi )^{p} = \frac {1}{\vol{E}}\int _{E} \rho _{E-x}(
\xi )^{p} \dvtex x
\end{equation*}
for $p\ne 0$ and as
\begin{equation*}
\log (\rho _{\omr 0 E}(\xi )) = \frac {1}{\vol{E}}\int _{E} \log (
\rho _{E-x}(\xi )) \dvtex x.
\end{equation*}
They showed that $\omr p E$ is a star body for $p>-1$ and a convex body
for $p\ge 0$.

Let $0<s<1$. For a star body $K\subset \mathbb R^{n}$ and a convex body
$E\subset \mathbb R^{n}$, we obtain from Fubini's theorem that
%
\begin{align}
\int _{E} \int _{E^{c}} \frac{1}{\Vert x-y\Vert _{K}^{n+s}}\dvtex x\dvtex y &=
\int _{E} \int _{(E-x)^{c}} \frac{1}{\Vert z\Vert _{K}^{n+s}}\dvtex z\dvtex x
\nonumber \\
\label{eq_Rms}
&=\frac{1}{s} \int _{E} \int _{\snvtex} \rho _{K}(\xi )^{n+s}\rho _{E-x}(
\xi )^{-s}\dvtex \xi \dvtex x
\\
&=\frac{\vol{E}}{s} \int _{\snvtex} \rho _{K}(\xi )^{n+s}\rho _{\omr {-s} E}(
\xi )^{-s}\dvtex \xi .
\nonumber \end{align}
Hence, by~\eqref{eq_fracperi_proj},
%
\begin{equation}
\tilde V_{-s}(K, \ofpp{s} E) =\frac {2\vol{E}}{s}\, \tilde V_{-s}(K,
\omr {-s} E)
\nonumber \end{equation}
for every star body $K$. By the equality case of the dual mixed volume
inequality~\eqref{eq_mixedvolume}, it follows that
%
\begin{equation}
\ofpp{s} E =\Big(\frac {s}{2\vol{E}}\Big)^{\frac{1}{s}} \omr {-s} E
\nonumber \end{equation}
for every convex body $E\subset \mathbb R^{n}$.

Hence, we can reformulate the Gardner--Zhang inequalities~\eqref{eq_gardnerzhang} obtained in \cite{GZ} to get inequalities for fractional
polar projection bodies. In particular, Theorem~5.5 and Lemma 5.7 from
\cite{GZ} imply that
%
\begin{equation}
\Big(\frac{\vol{\ofpp{s} E}}{\vol{\ofpp{s} \Svtex}}\Big)^{- s/ n} \le
\Big(\frac{\vol{ E}}{\vol{\Svtex}}\Big)^{ {(n-s)}/n}
\nonumber \end{equation}
for $E\subset \mathbb R^{n}$ a convex body and $0<s<1$ with equality precisely
if $E$ is a simplex, where $\Svtex $ is any $n$-dimensional simplex in
$\mathbb R^{n}$. Letting $s\to 1^{-}$ and using Theorem~\ref{thm_limitM},
we obtain the Zhang projection inequality \cite{Zhang91} (without the equality
case), which was reproved in \cite{GZ}:
%
\begin{equation}
\Big(\frac{\vol{\opp E}}{\vol{\opp \Svtex}}\Big)^{- 1/ n} \le \Big(
\frac{\vol{ E}}{\vol{\Svtex}}\Big)^{ {(n-1)}/n}
\nonumber \end{equation}
for $E\subset \mathbb R^{n}$ a convex body.

Conversely, we can reformulate Theorem~\ref{thm_fracpetty} for radial mean
bodies and obtain sharp affine isoperimetric inequalities for radial
$p$-th mean bodies for $-1<p<0$. Using a variation of our approach to Theorem~\ref{thm_fracpetty}, we obtain the following result. We have included the
case $p=n$ from \cite[Lemma 5.7]{GZ} in the statement of the result.

\begin{theorem}%
\label{thm_meanradial}
If $\,E\subset \mathbb R^{n}$ is a convex body, then
%
\begin{equation}
\label{eq_inequalities}
\begin{array}{r@{\hspace*{3pt}}c@{\hspace*{3pt}}ll}
\displaystyle \frac{\vol{\omr{p} E}}{\vol{ E}} &\le &\displaystyle
\frac{\vol{\omr{p} \Bvtex}}{\vol{\Bvtex}}&\quad \text{ for } \,-1<p<n,
\\[12pt]
\displaystyle \frac{\vol{\omr{p} E}}{\vol{ E}}&\ge &\displaystyle
\frac{\vol{\omr{p} \Bvtex}}{\vol{\Bvtex}}&\quad \text{ for }\,p>n,
\end{array}
\end{equation}
with equality if and only if $E$ is an ellipsoid. Here,
%
\begin{equation}
\frac{\vol{\omr{p} \Bvtex}}{\vol{\Bvtex}}=\displaystyle \Big(
\frac{2^{p+1}\omega _{n+p}}{(p+1)\omega _{n}\omega _{p+1}}\Big)^{n/p}
\nonumber \end{equation}
for $p>-1$ and $p\ne 0,n$, and
\begin{equation*}
\frac{\vol{\omr 0 \Bvtex}}{\vol{\Bvtex}}=2^{n} e^{\frac {n}{2}(\psi (
\frac{1}{2})-\psi (\frac{n}{2}+1))},
\end{equation*}
where $\psi $ is the digamma function. Moreover,
\begin{equation*}
\frac{\vol{\omr n E}}{\vol{E}}=1
\end{equation*}
for every convex body $E\subset \mathbb R^{n}$.
\end{theorem}

\eject

\begin{proof}
The inequality, including the equality case, for $-1<p<0$ is covered in
Theorem~\ref{thm_fracpetty}. We consider the cases $p\ge 0$. Let
$K\subset \mathbb R^{n}$ be a star body. For $p>0$, it follows from Fubini's
theorem that
%
\begin{equation}
\begin{split} \int _{E} \int _{E} \frac{1}{\Vert x-y\Vert _{K}^{n-p}}
\dvtex x\dvtex y &=\int _{E} \int _{E-x} \frac{1}{\Vert z\Vert _{K}^{n-p}}\dvtex z
\dvtex x
\\
&=\frac{1}{p} \int _{E} \int _{\snvtex} \rho _{K}(\xi )^{n-p} \rho _{E-x}(
\xi )^{p}\dvtex \xi \dvtex x
\\
&=\frac{\vol{E}}{p} \int _{\snvtex} \rho _{K}(\xi )^{n-p} \rho _{\omr {p} E}(
\xi )^{p} \dvtex \xi
\\
&=\frac{n \vol{E}}{p}\, \tilde V_{p}(K,\omr p E)
\label{eq_mixedvolume_RpE}
\end{split}
\end{equation}
for every convex body $E\subset \mathbb R^{n}$.

For $0<p<n$, as in the proof of Theorem~\ref{thm_perburchard}, we write
%
\begin{multline}
\int _{E}\int _{E} \frac{1}{\|x-y\|^{n-p}_{K}} \dvtex x \dvtex y
\\
= \int _{0}^{\infty }\int _{\mathbb R^{n}}\int _{\mathbb R^{n}} \chi _{E}(x)
\chi _{E}(y) \chi _{t^{- 1/({n-p})}K}(x-y) \dvtex x \dvtex y \dvtex t.
\nonumber \end{multline}
The Riesz rearrangement inequality, Theorem~\ref{thm_BLL}, and Theorem~\ref{thm_burchard} applied to the inner integral imply that
%
\begin{equation}
\int _{E} \int _{E} \frac{1}{\Vert x-y\Vert _{K}^{n-p}}\dvtex x\dvtex y\le
\int _{E^{\star}} \int _{E^{\star}}
\frac{1}{\Vert x-y\Vert _{K^{\star}}^{n-p}}\dvtex x\dvtex y,
\nonumber \end{equation}
with equality precisely if $E$ is equivalent to an ellipsoid homothetic
to~$K$. By~\eqref{eq_mixedvolume_RpE}, this is equivalent to
%
\begin{equation}
\label{eq_rieszee1}
\tilde V_{p}(K, \omr p E)
\le \tilde V_{p}(K^{\star}, \omr p E^{\star}),
\end{equation}
with equality precisely if $E$ is equivalent to an ellipsoid homothetic
to~$K$. Setting $K=\omr p E$, we obtain from~\eqref{eq_rieszee1} and the
dual mixed volume inequality~\eqref{eq_mixedvolume} for $0<p<n$ that
%
\begin{align}
\vol{\omr {p} E} &=\tilde V_{p}(\omr p E, \omr p E)
\nonumber \\
&\le \tilde V_{p}((\omr p E)^{\star}, \omr p (E^{\star}))
\nonumber \\
&\le \vol{\omr p E}^{\frac{n-p}{n}} \vol{\omr p (E^{\star})}^{
\frac {p}{n}}.
\nonumber \end{align}
Hence,
%
\begin{equation}
\label{eq_dmv1}
\vol{\omr {p} E}\le \vol{\omr {p} (E^{\star})}
\end{equation}
with equality if and only if $E$ is equivalent to an ellipsoid.

\eject

For $p>n$, let $c=\sup \{\Vert x-y\Vert _{K}: x,y\in E\}$. Note that
$c<\infty $ for a convex body $E\subset \mathbb R^{n}$. We write
%
\begin{align}
\int _{E}\int _{E} \frac{1}{\|x-y\|^{n-p}_{K}} &\dvtex x \dvtex y
\nonumber \\
&= c\, \vol{E}^{2} - \int _{E}\int _{E} (c-
\frac{1}{\|x-y\|^{n-p}_{K}}) \dvtex x \dvtex y
\nonumber \\
&= c\, \vol{E}^{2} - \int _{0}^{c} \int _{E}\int _{E} \chi _{(c-t)^{1/({p-n})}K}(x-y)
\dvtex x \dvtex y\dvtex t.
\nonumber \end{align}
Applying the Riesz rearrangement inequality, Theorem~\ref{thm_BLL}, and
Theorem~\ref{thm_burchard} to the inner integral and using~\eqref{eq_mixedvolume_RpE}, we obtain that
%
\begin{equation}
\label{eq_rieszee2}
\tilde V_{p}(K, \omr p E)
\ge \tilde V_{p}(K^{\star}, \omr p E^{\star})
\end{equation}
holds with equality precisely if $E$ is equivalent to an ellipsoid homothetic
to $K$. Setting $K=\omr p E$, we obtain from~\eqref{eq_rieszee2} and the
dual mixed volume inequality~\eqref{eq_mixedvolume} for $p>n$ that
%
\begin{align}
\vol{\omr {p} E} &=\tilde V_{p}(\omr p E, \omr p E)
\nonumber \\
&\ge \tilde V_{p}((\omr p E)^{\star}, \omr p (E^{\star}))
\nonumber \\
&\ge \vol{\omr p E}^{\frac{n-p}{n}} \vol{\omr p (E^{\star})}^{
\frac {p}{n}}.
\nonumber \end{align}
Hence,
%
\begin{equation}
\label{eq_dmv2}
\vol{\omr {p} E}\ge \vol{\omr {p} (E^{\star})}
\end{equation}
with equality if and only if $E$ is equivalent to an ellipsoid.

Next, we consider the case $p=0$. Let $K\subset \mathbb R^{n}$ be a star
body. Using definition~\eqref{eq_dVlog}, we obtain from Fubini's theorem
that
%
\begin{align}
n &\vol{E}\vol{K}\, \dVlog (K,\omr0 E)
\nonumber \\
&= \vol{E}\int _{\snvtex} \rho _{K}(\xi )^{n}\Big(\frac {1}{\vol{E}}
\int _{E} \log (\rho _{E-x}(\xi )) - \log (\rho _{K}(\xi )) \dvtex x\Big)
\dvtex \xi
\nonumber \\
&= \int _{E} \int _{\snvtex} \rho _{K}(\xi )^{n} \int _{\rho _{K}(\xi )}^{
\rho _{E-x}(\xi )} \frac {1}{r} \dvtex r \dvtex \xi \dvtex x
\nonumber \\
&= \int _{E} \int _{(E-x) \setminus K} \rho _{K}(z)^{n} \dvtex z \dvtex x-
\int _{E} \int _{K \setminus (E-x)} \rho _{K}(z)^{n} \dvtex z \dvtex x
\nonumber \\
&=\int _{E} \int _{E-x} \chi _{[0,1)}(\rho _{K}(z))\,\rho _{K}(z)^{n}
\dvtex z \dvtex x
\nonumber \\
&
\hskip 60pt
- \int _{E} \int _{(E-x)^{c}} \chi _{[1,\infty )}(\rho _{K}(z))\,
\rho _{K}(z)^{n} \dvtex z \dvtex x
\nonumber \\
&=\int _{E}\int _{E}
\frac{\chi _{(1,\infty )}(\|x-y\|_{K}) }{\|x-y\|_{K}^{n}} \dvtex x\dvtex y -
\int _{E}\int _{E^{c}}
\frac{\chi _{[0,1]}(\|x-y\|_{K})}{\|x-y\|_{K}^{n}}\dvtex x\dvtex y
\nonumber\\
&=\int_E\int_E \min\biggl\{1, \frac1{\|x-y\|_K^{n}}\biggr\} \dvtex x\dvtex y
\nonumber \\
&\quad {} - \int_E\int_{E^c} \max \biggl\{0, \frac1{\|x-y\|_K^{n}}-1\biggr\}\dvtex x\dvtex y -\vol{K}\,\vol{E}
\nonumber
\end{align}
for every convex body $E\subset \mathbb R^{n}$. Similar to the case
$0<p<n$ and to the proof of Theorem~\ref{thm_perburchard}, the Riesz rearrangement
inequality, Theorem~\ref{thm_BLL}, and Theorem~\ref{thm_burchard} show
that
%
\begin{equation}
\label{eq_rieszee3}
\dVlog (K,\omr0 E) \leq \dVlog (K^{\star},\omr 0 (E^{\star})),
\end{equation}
with equality if and only if $E$ is equivalent to an ellipsoid homothetic
to $K$. Setting $K=\omr 0 E$, we obtain from~\eqref{eq_rieszee3} and the
logarithmic dual mixed volume inequality~\eqref{eq_mixedvolume_log} that
%
\begin{align}
0 = \dVlog (\omr0 E,\omr0 E) \leq \dVlog ((\omr0 E)^{\star},\omr0 (E^{
\star})) \leq%
\frac {1}{n} \log \Big(\frac{\vol{\omr0(E^{\star})}}{\vol{\omr0E}}
\Big)
\nonumber \end{align}
and deduce that
%
\begin{equation}
\label{eq_dmv3}
\vol{\omr0 E} \leq \vol{\omr0(E^{\star})}
\end{equation}
with equality if and only if $E$ is equivalent to an ellipsoid.

Since $\omr p$ is positively homogeneous of degree 1, the inequalities
in~\eqref{eq_inequalities}, including the equality cases for
$0< p<n$, $p>n$ and $p=0$ follow from~\eqref{eq_dmv1},~\eqref{eq_dmv2}, and~\eqref{eq_dmv3}, respectively.

Finally, we calculate the constants. For $p>0$, we obtain by~\eqref{eq_mixedvolume_RpE} that
%
\begin{equation}
\label{eq_fracomr}
\Big(\frac{\vol{\omr p \Bvtex}}{\vol{\Bvtex}}\Big)^{1/n} = \Big(
\frac {p}{n\omega _{n}^{2}}\int _{\Bvtex }\int _{\Bvtex}
\frac{1}{\vert x-y\vert ^{n-p}}\dvtex x\dvtex y\Big)^{1/p}
\end{equation}
and by \cite[Theorem 7.2.7 and Theorem 8.6.6]{SchneiderWeil} that
%
\begin{equation}
\label{eq_riesz}
\int _{\Bvtex }\int _{\Bvtex} \frac{1}{\vert x-y\vert ^{n-p}}\dvtex x\dvtex y=
\frac{2^{p+1}n\omega _{n}\omega _{n+p}}{p(p+1)\omega _{p+1}}.
\end{equation}
For $-1<p<0$, we obtain by~\eqref{eq_Rms} that
%
\begin{equation}
\Big(\frac{\vol{\omr p \Bvtex}}{\vol{\Bvtex}}\Big)^{1/n} = \Big(
\frac {-p}{n\omega _{n}^{2}}\int _{\Bvtex }\int _{(\Bvtex )^{c}}
\frac{1}{\vert x-y\vert ^{n-p}}\dvtex x\dvtex y\Big)^{1/p}
\nonumber \end{equation}
and obtain the result from~\eqref{eq_psb}. Using~\eqref{eq_fracomr} and~\eqref{eq_riesz}, we obtain that
%
\begin{align}
\Big(\frac{\vol{\omr 0 \Bvtex}}{\vol{\Bvtex}}\Big)^{1/n} &=\lim _{p\to 0^{+}}
\Big(\frac{2^{p+1}\omega _{n+p}}{(p+1)\omega _{n}\omega _{p+1}}\Big)^{1/p}
= {2\, e^{\frac{1}{2}(\psi (\frac{1}{2})-\psi (\frac{n}{2}+1))}}.
\nonumber \end{align}
The case $p=n$ follows from the polar coordinate formula of volume combined
with the definition of the radial $n$-th mean body.
\end{proof}

\end{document}